\documentclass{article}
\usepackage{color}
\usepackage{eurosym}
\usepackage{amsmath}
\usepackage{amsfonts}
\usepackage{verbatim}
\usepackage{graphicx}

\newtheorem{theorem}{Theorem}

\newtheorem{corollary}[theorem]{Corollary}

\newtheorem{definition}[theorem]{Definition}

\newtheorem{lemma}[theorem]{Lemma}

\newtheorem{proposition}[theorem]{Proposition}
\newtheorem{remark}[theorem]{Remark}

\newenvironment{proof}[1][Proof]{\noindent\textbf{#1.} }{\ \rule{0.5em}{0.5em}}

\input{tcilatex}
\begin{document}

\title{Quantitative statistical stability, speed of convergence to
equilibrium and partially hyperbolic skew products. }
\author{Stefano Galatolo \thanks{
Dipartimento di Matematica, Universita di Pisa, Via \ Buonarroti 1,Pisa -
Italy. Email: galatolo@dm.unipi.it}}
\maketitle

\begin{abstract}
We consider a general relation between fixed point stability of suitably
perturbed transfer operators and convergence to equilibrium (a notion which
is strictly related to decay of correlations). We apply this relation to
deterministic perturbations of a class of (piecewise) partially hyperbolic
skew products whose behavior on the preserved fibration is dominated by the
expansion of the base map. In particular we apply the results to power law
mixing toral extensions. It turns out that in this case, the dependence of
the physical measure on small deterministic perturbations, in a suitable
anisotropic metric is at least H\"{o}lder continuous, with an exponent which
is explicitly estimated depending on the arithmetical properties of the
system. We show explicit examples of toral extensions having actually H\"{o}%
lder stability and non differentiable dependence of the physical measure on
perturbations.
\end{abstract}

\section{Introduction}

The concept of statistical stability of a dynamical system deals with the
stability of the statistical properties of its trajectories when the system
is perturbed or changed in some way. Since the statistical properties of
systems and their behavior are important in many fields of mathematics and
in applied science, the study of this kind of stability has important
applications in many fields. Many important statistical properties of
dynamics are encoded in suitable probability measures which are invariant
for the action of the dynamics; because of this the mathematical approach to
statistical stability is often related to the stability of such invariant
measures under perturbations of the system. In this context it is important
to get quantitative estimates, such as the differentiability of the
statistical properties under small perturbations (called Linear Response) or
other quantitative statements, such as the Lipschitz or H\"{o}lder
dependence. These questions are well understood in the uniformly hyperbolic
case, where the system's derivative has uniformly expanding or contracting
directions. In this case quantitative estimates are available, proving the
Lipschitz and even differentiable dependence of the relevant invariant
measures under perturbations of the system (see e.g \cite{ASsu}, \cite{BB}
or \cite{L2} and related references, for recent surveys where also some
result beyond the uniformly hyperbolic case are discussed). For systems
having not a uniformly hyperbolic behavior, and in presence of
discontinuities, the situation is more complicated and much less is known.
Qualitative results and some quantitative ones (providing precise
information on the modulus of continuity) are known under different
assumptions or in families of cases, and there is not yet a general
understanding of the statistical stability in those cases (see. e.g.\ \cite%
{A}, \cite{AV}, \cite{BV}, \cite{BT}, \cite{BBS}, \cite{BKL}, \cite{BS}, 
\cite{Dol}, \cite{D2}, \cite{Gmann}, \cite{Ko}, \cite{SV}, \cite{zz}).

We approach this question from a general point of view, using a functional
analytic perturbation lemma (see Theorem \ref{gen}) which relates the
convergence to equilibrium speed of the system to the stability of its
invariant measures belonging to suitable spaces. In our case we consider a
space of signed measures equipped with a suitable anisotropic norm adapted
to the system, in which the relevant invariant measures are proved to exist.
We show how, with some technical work, the approach can be applied to slowly
mixing partially hyperbolic skew products. The functional analytic
perturbation lemma we use is quite flexible and was applied in \cite{Gmann}
to the study of the statistical stabiliy of maps with indifferent fixed
points. A similar functional analytic construction was also applied to
piecewise hyperbolic skew product maps and Lorenz like two dimensional maps
in \cite{GL}.

\noindent \textbf{Main results. }The paper has the following structure and
main contents:

\begin{description}
\item[a)] A general quantitative relation between speed of convergence to
equilibrium of the system and its statistical stability (Section \ref{1}).

\item[b)] The application of this relation to a general class of skew
products allowing discontinuities and a sort of partial hyperbolicity,
getting quantitative estimates for their statistical stability in function
of their convergence to equilibrium. Here a main ingredient is the
construction of suitable spaces of regular measures adapted to these
systems. (Sections \ref{spaces},\ref{LY},\ref{Dist}).

\item[c)] The application of this construction to the stability under
deterministic perturbations of a class of piecewise constant \emph{toral
extensions having slow convergence to equilibrium}, getting H\"{o}lder
stability for these examples (Section \ref{Sec2}).

\item[d)] Finally, we show examples of mixing piecewise constant toral
extensions where a perturbation of the map of size $\delta $ results in a
change of physical invariant measure of the size of order $\delta ^{\beta }$%
, where $\beta \leq 1$ depends on the Diophantine properties of the map
(Section\ \ref{slat}).
\end{description}

\noindent Let us explain in more details the content of the items listed
above:

a) This relation is a fixed point stability statement we apply to the
system's transfer operators, giving nontrivial information for systems
having different kinds of speed of convergence to equilibrium, implying for
instance, that \emph{in a system with power law convergence to equilibrium
speed, under quite general additional assumptions, the physical measure is H%
\"{o}lder stable} (see Theorem \ref{gen} and Remark \ref{rmkhl}).

b) We apply the relation to show a general quantitative stability statement
for perturbations of a class of skew products. We consider skew products in
which the\emph{\ base dynamics is expanding and dominates the behavior of
the dynamics on the fibers}. For this purpose we introduce suitable spaces
of signed measures adapted to such systems. We consider spaces of signed
measures having absolutely continuous projection on the base space $[0,1]$
(corresponding to the strongly expanding direction) and equip them with
suitable anisotrpic norms: the weak norm $||~||_{"1"}$ and the strong norm $%
||~||_{p-BV}$, which are defined by disintegrating along the central
foliation (preserved by the skew product) and considering the regularity of
the disintegration. These spaces have properties that make them work quite
like $L^{1}$ and $p$-Bounded Variation real functions spaces in the
classical theory for the statistical properties of one dimensional dynamics.
In Section \ref{LY} we prove a kind of Lasota Yorke inequality in this
framework. This will be used together with a kind of Helly selection
principle proved in Section \ref{spaces} to estimate the regularity of the
invariant measures (like it is done in the classical construction for one
dimensional, piecewise expanding maps). \ We summarize this in the following
result (see Proposition \ref{regu} for a precise and more general statement).

\begin{theorem}
\label{28 copy(2)}Consider a family of skew product maps $F_{\delta
}:[0,1]\times M\rightarrow \lbrack 0,1]\times M,$ where $M$ is a compact
manifold with boundary, $F_{\delta }=(T_{\delta },G_{\delta })$ is such that
the following hold uniformly on $\delta $

\begin{description}
\item[A1] $T_{\delta }:[0,1]\rightarrow \lbrack 0,1]$ is piecewise expanding
with $C^{2}$ onto branches;

\item[A2] the behavior of $G_{\delta }$ on the fibers is dominated by the
expansion of $T_{\delta }$;

\item[A3] $G_{\delta }$ satisfies a sort of BV \ regularity: there is $A>0,$
such that%
\begin{equation*}
\sup_{r\leq A}\frac{1}{r}\int \sup_{y\in M,x_{1},x_{2}\in
B(x,r)}|G(x_{1},y)-G(x_{2},y)|dx<\infty
\end{equation*}
\end{description}

\noindent (see $Sk1,...,Sk3$ \ at beginning of Section \ref{spaces} for
precise statements of these assumptions). Then the maps $F_{\delta }$ have
invariant probability measures $f_{\delta }$ having an absolutely continuous
projection on the base space $[0,1]$ and uniformly bounded $||~||_{p-BV}$
norms (they are uniformly regular in the strong space).
\end{theorem}

In Section \ref{Dist} we consider a class of perturbations of our skew
products such that the related transfer operators are near in some sense
when applied to (regular) measures and state a first \emph{general statement
on the statistical stability of such skew products} (see Proposition \ref%
{thm}).

c) The statement is then applied to \emph{slowly mixing piecewise constant
toral extensions}: systems of the kind $(X,F)$ where $X=[0,1]\times \mathcal{%
T}^{d}$, $\mathcal{T}^{d}$ is the $d$ dimensional torus and $F:X\rightarrow X
$ is defined by%
\begin{equation}
F(\omega ,t)=(T\omega ,t+\mathbf{\tau }(\omega ))  \label{skewsimpl}
\end{equation}%
where $T:[0,1]\rightarrow \lbrack 0,1]$ is expanding and $\tau
:[0,1]\rightarrow \mathcal{T}^{d}$ is a piecewise constant function. The
qualitative ergodic theory of this kind of systems was studied in several
papers (see e.g. \cite{Br},\cite{BW}). Quantitative results appeared more
recently (\cite{Do}, \cite{BE},\cite{Na}), proving from different points of
view that the speed of correlation decay is generically fast\ (exponential),
but \ in some cases where $\tau $ is piecewise constant, this decay follows
a power law whose exponent depend on the diophantine properties of $\mathbf{%
\tau }$ (see \cite{GSR} or Section \ref{decorr1}).

We apply our general result to deterministic perturbations of these maps,
showing that the physical measure of those systems varies at least H\"{o}%
lder continuously in our anisotropic "$L^{1}$ like" distance. We state
informally an example of such an application (see Proposition \ref{28} for a
more general statement and the required definitions).

\begin{theorem}
\label{28 copy(1)}Consider a family of skew product maps $F_{\delta
}:[0,1]\times \mathcal{T}^{d}\rightarrow \lbrack 0,1]\times \mathcal{T}^{d}$
satisfing A1,...,A3, as in Theorem \ref{28 copy(2)}. Let us assume $F_{0}$
is a piecewise constant toral extension as in $(\ref{skewsimpl})$, with%
\begin{equation}
T_{0}(x)=2x~\func{mod}(1)  \label{2x}
\end{equation}%
and%
\begin{equation}
G_{0}(x,t)=(Tx,t+\theta \varphi (x))  \label{22x}
\end{equation}%
where $\theta =(\theta _{1},...,\theta _{d})\in \mathcal{T}^{d}$ has linear
Diophantine type\footnote{%
See Definition \ref{linapp} for a recall about this Diophantine type for
vectors of real numbers.} $\gamma _{l}(\theta )$ and $\varphi =1_{[0,\frac{1%
}{2}]}$ is the characteristic function of $[0,\frac{1}{2}]\footnote{%
For such map the Lebesgue measure $f_{0}$ is invariant for the system.}$.
Suppose $F_{\delta }$ is a small perturbation of $F_{0}$ in the following
sense

\begin{description}
\item[D1] for each $\delta ,~T_{\delta }=T_{0}\circ \sigma $ for some
diffeomorphism $\sigma $ near to the identity $||\sigma -Id||_{\infty }\leq
\delta ,||\frac{1}{\sigma ^{\prime }}-1||_{\infty }\leq \delta $;

\item[D2] for each $\delta $ and $x\in \lbrack 0,1],~y\in \mathcal{T}^{d}:$ $%
|G_{0}(x,y)-G_{\delta }(x,y)|\leq \delta .$
\end{description}

Then and for each $\gamma >\gamma _{l}(\theta )$ there is $K\geq 0$ such
that,%
\begin{equation*}
||f_{\delta }-f_{0}||_{"1"}\leq K\delta ^{\frac{1}{8\gamma +1}}.
\end{equation*}
\end{theorem}

We remark that the perturbations allowed are quite general. In particular
they allow discontinuities, and the invariant measure to become singular
with respect to the Lebesgue measure after perturbation. We also remark that
for a class of smooth toral extensions with fast decay of correlations, a
differentiable dependence statement was proved in \cite{D2}.

d) We finally show examples of piecewise constant, mixing toral extensions
where the physical measure of the system actually varies in a H\"{o}lder way
(and hence not in a differentiable way) with an exponent depending on the
arithmetical properties of the system. We state informally the main result
abou this, see Propositions \ref{bahh} and \ref{30} for precise statements.

\begin{theorem}
Let us consider a piecewise constant toral extension map $F_{0}:[0,1]\times 
\mathcal{T}^{1}$, as in $(\ref{2x}),(\ref{22x})$, where $\theta $ is a well
approximable Diophantine irrational with $\gamma _{l}(\theta )>2$. For every 
$\gamma <\gamma _{l}(\theta )$ there is a sequence of reals $\delta _{j}\geq
0$, $\delta _{j}\rightarrow 0$ and a sequence of maps $\hat{F}_{\delta
_{j}}(x,y)=(\hat{T}_{0}(x),\hat{G}_{\delta _{j}}(x,y))$ satisfying
A1,...,A3, D1,D2 such that 
\begin{equation*}
||\mu _{0}-\mu _{j}||_{"1"}\geq \frac{1}{9}\delta _{j}{}^{\frac{1}{\gamma -1}%
}
\end{equation*}%
holds for every $j$ and every $\mu _{j}$, invariant borel probability
measure of $\hat{F}_{\delta _{j}}$ with absolutely continuous projection on $%
[0,1]$.
\end{theorem}

This shows that in some sense, the general statistical stability result is
sharp. We remark that recently, in \cite{zz} examples of $C^{r}$ families of 
\emph{mostly contracting} diffeomorphisms with strictly H\"{o}lder \
behavior have been given (see also \cite{Dol} for previous results on H\"{o}%
lder stability of these kinds of partially hyperbolic maps).

\noindent \textbf{Aknowledgements. }{The work was partially supported by EU
Marie-Curie IRSES Brazilian-European partnership in Dynamical Systems
(FP7-PEOPLE-2012-IRSES 318999 BREUDS), and by The Leverhulme Trust through
Network Grant IN-2014-021.}

\section{Quantitative fixed point stability and convergence to equilibrium. 
\label{1}}

Let us consider a dynamical system $(X,T)$ where $X$ is a metric space and
the space $SM(X)$ of signed Borel measures on $X$. The dynamics $T$
naturally induces a function $L:SM(X)\rightarrow SM(X)$ which is linear and
is called transfer operator. If $\nu \in SM(X)$ then $L[\nu ]\in SM(X)$ is
defined by%
\begin{equation*}
L[\nu ](B)=\nu (T^{-1}(B))
\end{equation*}%
for every measurable set $B$. If $X$ is a manifold, the measure is
absolutely continuous ($d\nu =f~dm$, where $m$ represents the Lebesgue
measure) and \thinspace $T$ is nonsingular, the operator induces another
operator $\tilde{L}:L^{1}(m)\rightarrow L^{1}(m)$ acting on measure
densities ($\tilde{L}f=\frac{d(L(f~m))}{dm}$). By a small abuse of notation
we will still indicate by $L$ this operator.

An invariant measure is a fixed point for the transfer operator. Let us now
see a quantitative stability statement for these fixed points under suitable
perturbations of the operator. \ Let us consider a certain system having a
transfer operator $L_{0}$ for which we know the speed of convergence to
equilibrium (see (\ref{equil}) below). Consider a "nearby" system $L_{1}$
having suitable properties: suppose there are two normed vector spaces of
measures with sign $B_{s}\subseteq B_{w}\subseteq SM(X)$ (the strong and
weak space) with norms $||~||_{w}\leq ||~||_{s}$ and suppose the operators $%
L_{0}$ and $L_{1}$ preserve the spaces: $L_{i}(B_{s})\subseteq B_{s}$ and $%
L_{i}(B_{w})\subseteq B_{w}$ with $i\in \{0,1\}$. Let us consider%
\begin{equation*}
V_{s}:=\{f\in B_{s},f(X)=0\}
\end{equation*}%
the space of zero average measures in $B_{s}$. The speed of convergence to
equilibrium of a system will be measured by the speed of contraction to $0$
of this space by the iterations of the transfer operator.

\begin{definition}
\label{d1}Let $\phi (n)$ be a real sequence converging to zero. We say that
the system has \emph{convergence to equilibrium }with respect to norms $%
||~||_{w}$, $||~||_{s}$ and speed $\phi $ if $\forall g\in V_{s}$%
\begin{equation}
||L_{0}^{n}(g)||_{w}\leq \phi (n)||g||_{s}.  \label{equil}
\end{equation}
\end{definition}

Suppose $f_{0}$, $f_{1}\in B_{s}$ are fixed probability measures of $L_{0}$
and $L_{1}$. The following statement relates the distance between $f_{0}$
and $f_{1}$ with the distance between $L_{0}$ and $L_{1}$ and the speed of
convergence to equilibrium of $L_{0}.$ The proof is elementary, we include
it for completeness. Similar quantitative stability statements are used in 
\cite{GN}, \cite{GL} and \cite{Gmann} to support rigorous computation of
invariant measures, get quantitative estimates for the statistical stability
of Lorenz like maps and intermittent systems.

\begin{theorem}
\label{gen}Suppose we have estimates on the following aspects of the
operators $L_{0}$ and $L_{1}$:

\begin{enumerate}
\item (speed of convergence to equilibrium) there is $\phi \in C^{0}(\mathbb{%
R)},~\phi (t)$ decreasing to $0$ as $t\rightarrow \infty $ such that $L_{0}$
has convergence to equilibrium with respect to norms $||~||_{w}$, $||~||_{s}$%
and speed $\phi $.

\item (control on \ the norms of the invariant measures) There is $\tilde{M}%
\geq 0$ such that%
\begin{equation*}
\max (||f_{1}||_{s},||f_{0}||_{s})\leq \tilde{M};
\end{equation*}

\item (iterates of the transfer operator are bounded for the weak norm)
there is $\tilde{C}\geq 0$ such that for each $n$, 
\begin{equation*}
||L_{0}^{n}||_{B_{w}\rightarrow B_{w}}\leq \tilde{C}.
\end{equation*}

\item (control on the size of perturbation in the strong-weak norm) Denote 
\begin{equation*}
\underset{||f||_{s}\leq 1}{\sup }||(L_{1}-L_{0})f||_{w}:=\epsilon
\end{equation*}%
consider the decreasing function $\psi $ defined as $\psi (x)=\frac{\phi (x)%
}{x}$, then \ we have the following explicit estimate%
\begin{equation*}
||f_{1}-f_{0}||_{w}\leq 2\tilde{M}\tilde{C}\epsilon (\psi ^{-1}(\frac{%
\epsilon \tilde{C}}{2})+1).
\end{equation*}
\end{enumerate}
\end{theorem}

\begin{proof}
The proof is a direct computation from the assumptions%
\begin{eqnarray*}
||f_{1}-f_{0}||_{w} &\leq &||L_{1}^{N}f_{1}-L_{0}^{N}f_{0}||_{w} \\
&\leq
&||L_{1}^{N}f_{1}-L_{0}^{N}f_{1}||_{w}+||L_{0}^{N}f_{1}-L_{0}^{N}f_{0}||_{w}
\\
&\leq &||L_{0}^{N}(f_{1}-f_{0})||_{w}+||L_{1}^{N}f_{1}-L_{0}^{N}f_{1}||_{w}.
\end{eqnarray*}%
Since $f_{1}-f_{0}\in V_{s}$ , $||f_{1}-f_{0}||_{s}\leq 2\tilde{M}$,%
\begin{equation*}
||f_{1}-f_{0}||_{w}\leq 2\tilde{M}\phi
(N)+||L_{1}^{N}f_{1}-L_{0}^{N}f_{1}||_{w}
\end{equation*}%
but%
\begin{equation*}
(L_{0}^{N}-L_{1}^{N})=\sum_{k=1}^{N}L_{0}^{N-k}(L_{0}-L_{1})L_{1}^{k-1}
\end{equation*}%
hence%
\begin{eqnarray*}
-(L_{1}^{N}-L_{0}^{N})f_{1}
&=&\sum_{k=1}^{N}L_{0}^{N-k}(L_{0}-L_{1})L_{1}^{k-1}f_{1} \\
&=&\sum_{k=1}^{N}L_{0}^{N-k}(L_{0}-L_{1})f_{1}
\end{eqnarray*}%
then%
\begin{equation*}
||f_{1}-f_{0}||_{w}\leq 2\tilde{M}\phi (N)+\epsilon \tilde{M}N\tilde{C}.
\end{equation*}

Now consider the function $\psi $ defined as $\psi (x)=\frac{\phi (x)}{x},$
choose $N$ such that $\psi ^{-1}(\frac{\epsilon \tilde{C}}{2})\leq N\leq
\psi ^{-1}(\frac{\epsilon \tilde{C}}{2})+1$, in this way $\frac{\phi (N)}{N}%
\leq \frac{\epsilon \tilde{C}}{2}\leq \frac{\phi (N-1)}{N-1}$ and%
\begin{equation*}
||f_{1}-f_{0}||_{w}\leq 2\tilde{M}\tilde{C}\epsilon (\psi ^{-1}(\frac{%
\epsilon \tilde{C}}{2})+1).
\end{equation*}
\end{proof}

Theorem \ref{gen} \ implies that a system having convergence to equilibrium
is statistically stable (in the weak norm).

\begin{corollary}
\label{k1} Let $L_{\delta }$, $\delta \in \lbrack 0,\bar{\delta})$ be a
family of transfer operators under the assumptions of Theorem \ref{gen},
including $\lim_{n\rightarrow \infty }\phi (n)=0.$ Let $f_{\delta }$ be a \
fixed probability measure of $L_{\delta }.$ Suppose there is $C\geq 0$ such
that for every $\delta $ 
\begin{equation*}
\underset{||f||_{s}\leq 1}{\sup }||(L_{\delta }-L_{0})f||_{w}\leq C\delta .
\end{equation*}
Then $f_{0}$ is the unique fixed probability measure in $B_{s}$ and it holds%
\begin{equation*}
\lim_{\delta \rightarrow 0}||f_{\delta }-f_{0}||_{w}=0.
\end{equation*}
\end{corollary}

\begin{proof}
The uniqueness of $f_{0}$ is trivial from the definition of convergence to
equilibrium. For the stability, suppose there was a sequence $\delta
_{n}\rightarrow 0$ and $l\geq 0$ such that $||f_{\delta
_{n}}-f_{0}||_{w}\geq l$ $\forall n$, then 
\begin{eqnarray*}
2\tilde{M}\tilde{C}C\delta _{n}(\psi ^{-1}(\frac{\delta _{n}C\tilde{C}}{2}%
)+1) &\geq &l \\
\psi ^{-1}(\frac{\delta _{n}C\tilde{C}}{2}) &\geq &\frac{l}{\delta _{n}2%
\tilde{M}\tilde{C}C}-1 \\
\frac{\delta _{n}C\tilde{C}}{2} &\leq &\frac{\phi (\frac{l}{\delta _{n}2%
\tilde{M}\tilde{C}C}-1)}{\frac{l}{\delta _{n}2\tilde{M}\tilde{C}C}-1} \\
(\frac{l}{\delta _{n}2\tilde{M}\tilde{C}C}-1)\frac{\delta _{n}C\tilde{C}}{2}
&\leq &\phi (\frac{l}{\delta _{n}2\tilde{M}\tilde{C}C}-1) \\
\frac{l}{4\tilde{M}}-\delta _{n}C\tilde{C} &\leq &\phi (\frac{l}{\delta _{n}2%
\tilde{M}\tilde{C}C}-1)
\end{eqnarray*}%
which is impossible to hold as $\delta _{n}\rightarrow 0.$
\end{proof}

\begin{remark}
\label{rmkhl}In Theorem \ref{gen} , if $\phi (x)=Cx^{-\alpha }$ then $\psi
(x)=Cx^{-\alpha -1}$, $\epsilon (\psi ^{-1}(\epsilon )+1)\sim \epsilon ^{1-%
\frac{1}{\alpha +1}}$ and we have the estimate for the modulus of continuity%
\begin{equation*}
||f_{1}-f_{0}||_{w}\leq K_{1}\epsilon ^{1-\frac{1}{\alpha +1}}
\end{equation*}%
where the constant $K_{1}$ depends on $\tilde{M},\tilde{C},C$ and not on the
distance between the operators measured by $\epsilon $.
\end{remark}

\section{Spaces we consider\label{spaces}}

Our approach is based on the study of the transfer operator restricted to a
suitable space of measures with sign. We introduce a space of regular
measures where we can find the invariant measure of our systems, and the
ones of suitable perturbations of it. We hence consider some measure spaces
adapted to skew products. The approach is taken from \cite{GL} (see also 
\cite{AGP}) where it was used for Lorenz-like two dimensional maps. Let us
consider a map $F:X\rightarrow X$ where $X=[0,1]\times M,$ and $M$ is a
compact manifold with boundary, such that 
\begin{equation}
F(x,y)=(T(x),G(x,y)).  \label{1eq}
\end{equation}%
Suppose $F$ satisfies the following conditions:

\begin{description}
\item[Sk1] Suppose $T$ is $\frac{1}{\lambda }$-expanding\footnote{%
We suppose that $\underset{x\in \lbrack 0,1]}{\inf }T^{\prime }(x)\geq \frac{%
1}{\lambda }$ for some $\lambda <1.$} and it has $C^{1+\xi }$ branches%
\footnote{%
More precisely we suppose that there are $\xi $,$C_{h}\geq 0$ such that 
\begin{equation*}
\frac{1}{|T_{i}^{^{\prime }}\circ T_{i}^{-1}(\gamma _{2}))|}-\frac{1}{%
|T_{i}^{^{\prime }}\circ T_{i}^{-1}(\gamma _{1}))|}\leq C_{h}d(\gamma
_{1},\gamma _{2})^{\xi },\forall \gamma _{1},\gamma _{2}\in \lbrack 0,1].
\end{equation*}%
} which are onto. The branches will be denoted by $T_{i}$, $i\in \lbrack
1,..,q].$

\item[Sk2] Consider the $F$-invariant foliation $\mathcal{F}%
^{s}:=\{\{x\}\times M\}_{x\in \lbrack 0,1]}$. We suppose that the behavior
on $\mathcal{F}^{s}$ is dominated by $\lambda $: there is $\alpha \in 
\mathbb{R}$ with $\lambda ^{\xi }\alpha <1$, such that for all $x\in \lbrack
0,1]$ holds%
\begin{equation}
|G(x,y_{1})-G(x,y_{2})|\leq \alpha |y_{1}-y_{2}|\ \ \mathnormal{for\ all}\ \
y_{1},y_{2}\in M.  \label{contracting1}
\end{equation}

\item[Sk3] For each $p\leq \xi $ there is $A>0,$ such that%
\begin{equation*}
\hat{H}:=\sup_{r\leq A}\frac{1}{r^{p}}\int \sup_{y\in M,x_{1},x_{2}\in
B(x,r)}|G(x_{1},y)-G(x_{2},y)|dx<\infty
\end{equation*}
\end{description}

\begin{remark}
\label{Kell}We remark that $Sk3$ \ allows discontinuities in $G,$ provided a
kind of bounded variation regularity is respected. $Sk2$ allows a dominated
expansion or contraction in the fibers direction. Furthermore, by $Sk1$ the
transfer operator of the map $T$ satisfies a Lasota Yorke inequality of the
kind 
\begin{equation}
||L_{T}^{n}(\mu )||_{BV}\leq A_{T}\lambda ^{n}||\mu ||_{BV}+B_{T}||\mu ||_{1}
\label{ly1d}
\end{equation}%
where $||\mu ||_{BV}$ is the generalized bounded variation norm (see \cite%
{Gk}); for some constant $A_{T}$ and $B_{T}$ depending on the map.
\end{remark}

\begin{definition}
We say that a family of maps $F_{\delta }=(T_{\delta }(x),G_{\delta }(x,y))$
satistifies $Sk1,...,Sk3$ \emph{uniformly}, if each $T_{\delta }$ is
piecewise expanding, with onto $C^{1+\xi }$ branches, admitting a uniform
expansion rate $\frac{1}{\lambda },$ a uniform $\alpha $, a uniform H\"{o}%
lder constant $C_{h}$, a uniform second coefficient of the Lasota Yorke
inequality $B_{T\delta }$ and furthermore the family $G_{\delta }$ satisfies 
$Sk3$ with a uniform bound on the constant $\hat{H}$.
\end{definition}

We construct now some function spaces which are suitable for the systems we
consider. The idea is to consider spaces of measures with sign, with
suitable norms constructed by disintegrating measures along the central
foliation. In this way a measure on $X$ will be seen as a collection (a
path) of measures on the leaves. In the central direction (on the leaves) we
will consider a norm which is the dual of the Lipschitz norm. In the
expanding direction we will consider the $L^{1}$ norm and a suitable
variation norm. These ideas will be implemented in the next paragraphs.

Let $(X,d)$ be a compact metric space, $g:X\longrightarrow \mathbb{R}$ be a
Lipschitz function and let $Lip(g)$ be its best Lipschitz constant, i.e. 
\begin{equation*}
\displaystyle{Lip(g)=\sup_{x,y\in X}\left\{ \dfrac{|g(x)-g(y)|}{d(x,y)}%
\right\} }.
\end{equation*}

\begin{definition}
Given two signed Borel measures $\mu $ and $\nu $ on $X,$ we define a 
\textbf{Wasserstein-Kantorovich Like} distance between $\mu $ and $\nu $ by%
\begin{equation}
W_{1}(\mu ,\nu )=\sup_{Lip(g)\leq 1,||g||_{\infty }\leq 1}\left\vert \int {g}%
d\mu -\int {g}d\nu \right\vert .
\end{equation}%
\label{wasserstein}
\end{definition}

Let us denote%
\begin{equation*}
||\mu ||_{W_{1}}:=W_{1}(0,\mu ).
\end{equation*}%
As a matter of fact, $||\cdot ||_{W_{1}}$ defines a norm on the vector space
of signed measures defined on a compact metric space.

Let $\mathcal{SB}(\Sigma )$ be the space of Borel signed measures on $\Sigma$%
. Given $\mu \in \mathcal{SB}(\Sigma )$ denote by $\mu ^{+}$ and $\mu ^{-}$
the positive and the negative parts of it ($\mu =\mu ^{+}-\mu ^{-}$).

Denote by $\mathcal{AB}$ the set of signed measures $\mu \in \mathcal{SB}%
(\Sigma )$ such that its associated marginal signed measures, $\mu _{x}^{\pm
}=\pi _{x}^{\ast }\mu ^{\pm }$ are absolutely continuous with respect to the
Lebesgue measure $m$, on $[0,1]$ i.e.%
\begin{equation*}
\mathcal{AB}=\{\mu \in \mathcal{SB}(\Sigma ):\pi _{x}^{\ast }\mu ^{+}<<m\ \ 
\mathnormal{and}\ \ \pi _{x}^{\ast }\mu ^{-}<<m\}
\end{equation*}%
where $\pi _{x}:X\longrightarrow \lbrack 0,1]$ is the projection defined by $%
\pi _{x}(x,y)=x$ and \ $\pi _{x}^{\ast }$ is the associated pushforward map.

Let us consider a finite positive measure $\mu \in \mathcal{AB}$ \ on the
space $X$ foliated by the preserved leaves $\mathcal{F}^{c}=\{\gamma
_{l}\}_{l\in \lbrack 0,1]}$ such that $\gamma _{l}={\pi _{x}}^{-1}(l)$. We
will also call $\mathcal{F}^{c}$ as the \emph{central foliation. }Let us
denote $\mu _{x}=\pi _{x}^{\ast }\mu $ and let $\ \phi _{\mu }$ be its
density ($\mu _{x}=\phi _{\mu }m$ ). The Rokhlin disintegration theorem
describes a disintegration of $\mu $ by a family $\{\mu _{\gamma }\}_{\gamma
}$ of probability measures on the central leaves\footnote{%
In the following to simplify notations, when no confusion is possible we
will indicate the generic leaf or its coordinate with $\gamma $.} in a way
that the following holds.

\begin{remark}
The disintegration of a measure $\mu $ is the $\mu _{x}$-unique measurable
family $(\{\mu _{\gamma }\}_{\gamma },\phi _{\mu })$ such that, for every
measurable set $E\subset X$ it holds 
\begin{equation}
\mu (E)=\int_{[0,1]}{\mu _{\gamma }(E\cap \gamma )}d\mu _{x}(\gamma ).
\label{eee6}
\end{equation}%
\label{rmkv}
\end{remark}

\begin{definition}
Let $\pi _{\gamma ,y}:\gamma \longrightarrow M$ be the restriction $\pi
_{y}|_{\gamma }$, where $\pi _{y}:X\longrightarrow M$ is the projection
defined by $\pi _{y}(x,y)=y$ and $\gamma \in \mathcal{F}^{c}$. Given a
positive measure $\mu \in \mathcal{AB}$ and its disintegration along the
stable leaves $\mathcal{F}^{c}$, $\left( \{\mu _{\gamma }\}_{\gamma },\phi
_{\mu }\right) $, we define the \textbf{restriction of $\mu $ on $\gamma $}
as the positive measure $\mu |_{\gamma }$ on $M$ (not on the leaf $\gamma $)
defined as 
\begin{equation*}
\mu |_{\gamma }=\pi _{\gamma ,y}^{\ast }(\phi _{\mu }(\gamma )\mu _{\gamma
}).
\end{equation*}

\begin{definition}
For a given signed measure $\mu \in \mathcal{AB}$ and its decomposition $\mu
=\mu ^{+}-\mu ^{-}$, define the \textbf{restriction of $\mu $ on $\gamma $}
by%
\begin{equation}
\mu |_{\gamma }=\mu ^{+}|_{\gamma }-\mu ^{-}|_{\gamma }.
\end{equation}%
\label{restrictionmeasure}
\end{definition}

\begin{definition}
Let $\mathcal{L}^{1}\subseteq \mathcal{AB}$ be defined as%
\begin{equation}
\mathcal{L}^{1}=\left\{ \mu \in \mathcal{AB}:\int_{[0,1]}||\mu |_{\gamma
}||_{W_{1}}~dm(\gamma )<\infty \right\}  \label{L1measurewithsign}
\end{equation}%
and define norm on it, $||\cdot ||_{"1"}:\mathcal{L}^{1}\longrightarrow 
\mathbb{R}$, by%
\begin{equation}
||\mu ||_{"1"}=\int_{[0,1]}||\mu |_{\gamma }||_{W_{1}}~dm(\gamma ).
\label{l1normsm}
\end{equation}
\end{definition}
\end{definition}

The notation we use for this norm is similar to the usual $L^{1}$ norm.
Indeed this is formally the case if we associate to $\mu ,$ by
disintegration, a path $G_{\mu }:[0,1]\rightarrow (\mathcal{SB}%
(M),||~||_{W_{1}})$ defined by $\ G_{\mu }($ $\gamma )=\mu |_{\gamma }$. In
this case, this will be the $L^{1}$ norm of the path. \label{l1likespace}

\subsection{The transfer operator associated to $F$ and basic properties of $%
\mathcal{L}^{1}$}

Let us now consider the transfer operator $L_{F}$ associated with $F$. Being
a push forward map, the same function can be also denoted by $F^{\ast }$ we
will use this notation sometime. There is a nice characterization of the
transfer operator in our case, which makes it work quite like a one
dimensional operator. For the proof see \cite{GL}.

\begin{proposition}[Perron-Frobenius like formula]
\label{PF}Let us consider a skew product map $F$ satisfying $Sk1$ and $Sk2$.
For a given leaf $\gamma \in \mathcal{F}^{s}$, define the map $F_{\gamma
}:M\longrightarrow M$ by 
\begin{equation*}
F_{\gamma }=\pi _{y}\circ F|_{\gamma }\circ \pi _{\gamma ,y}^{-1}.
\end{equation*}%
For all $\mu \in \mathcal{L}^{1}$ and for almost all $\gamma \in \lbrack
0,1] $ it holds 
\begin{equation}
(L_{F}\mu )|_{\gamma }=\sum_{i=1}^{q}{\dfrac{\func{F}_{T_{i}^{-1}(\gamma
)}^{\ast }\mu |_{T_{i}^{-1}(\gamma )}}{|T_{i}^{^{\prime }}\circ
T_{i}^{-1}(\gamma ))|}}\ .  \label{niceformulaa}
\end{equation}%
\label{niceformulaab}
\end{proposition}

We recall some results showing that the transfer operator associated to a
Lipschitz function is also Lipschitz with the same constant, for the $"1"$
distance, and moreover, that the transfer operator of a map satisfying $%
Sk1,...,Sk3$ is also Lipschitz with the same constant for the $||~||_{"1"}$
norm. In particular, if $\alpha \leq 1$ the transfer operator is a weak
contraction, like it happen for the $L^{1}$ norm on the one dimensional case
(for the proof and more details see \cite{GL}).

\begin{lemma}
\label{unamis} If \ $G$ $:Y\rightarrow Y$, where $Y$ is a metric space is $%
\alpha $-Lipschitz, for every Borel measure with sign $\mu $ it holds%
\begin{equation*}
||L_{G}\mu ||_{W_{1}}\leq \alpha ||\mu ||_{W_{1}}.
\end{equation*}%
If $\mu \in \mathcal{L}^{1}$and $F:[0,1]\times M\rightarrow \lbrack
0,1]\times M$ satisfies $Sk1,...,Sk3$ then 
\begin{equation}
||L_{F}\mu ||_{"1"}\leq \alpha ||\mu ||_{"1"}.
\end{equation}
\end{lemma}

\subsection{The strong norm}

We consider a norm which is stronger than the $\mathcal{L}^{1}$ norm. The
idea is to consider a disintegrated measure as a path of measures on the
preserved foliation and define a kind of bounded variation regularity for
this path, in a way similar to what was done in \cite{Gk} for real functions.

For this strong space we will prove a regularization inequality, similar to
the Lasota Yorke ones. We will use this inequality to prove the regularity
of the invariant measure of the family of skew products we consider.

Let us consider $\mu \in \mathcal{L}^{1}.$ Let us define 
\begin{equation*}
osc(\mu ,x,r)=\limfunc{esssup}_{\gamma _{2},\gamma _{1}\in B(x,r)}(W_{1}(\mu
|_{\gamma _{1}},\mu |_{\gamma _{2}}))
\end{equation*}%
and

\begin{equation*}
var_{p}(\mu ,r)=\int_{[0,1]}r^{-p}osc(\mu ,x,r)~dx.
\end{equation*}%
Now \ let us \ choose $A>0$ and consider $var_{p}(\mu ):=\sup_{r\leq
A}var_{p}(\mu ,r)$. Finally we define $p-BV$ norm as:

\begin{equation*}
||\mu ||_{p-BV}=||\mu ||_{"1"}+var_{p}(\mu ).
\end{equation*}%
Let us consider $1\geq p\geq 0$ and the following space of measures%
\begin{equation*}
p-BV=\left\{ \mu \in \mathcal{L}^{1},||\mu ||_{p-BV}<\infty \right\} .
\end{equation*}

This will play the role of the strong space in the present case.\ 

\begin{remark}
\label{Ap}If $\mu \in p-BV$, then it follows that%
\begin{equation*}
\underset{\gamma }{ess\sup }||\mu |_{\gamma }||_{W_{1}}\leq A^{p-1}||\mu
||_{p-BV}.
\end{equation*}%
See \cite{AGP}, Lemma 2 for a proof in the case of real functions which also
works in our case.
\end{remark}

We now prove a sort of Helly selection principle for sequences of \emph{%
positive }measures with bounded variation. This principle will be used,
together with a regularization inequality, proved in next section, to get
information on the variation of invariant measures. First we need a
preliminary lemma:

\begin{lemma}
\label{spp}If $\mu _{n}$ is a sequence of positive measures such that for
each $n$, $||\mu _{n}||_{"1"}\leq C$, $var_{p}(\mu _{n})\leq M$, and $\mu
_{n}|_{\gamma }\rightarrow \mu |_{\gamma }$ for a.e. $\gamma ,$ in the
Wasserstein distance, then%
\begin{equation*}
||\mu ||_{"1"}\leq C,var_{p}(\mu )\leq M.
\end{equation*}
\end{lemma}

\begin{proof}
Let us consider the $||~||_{"1"}$ norm. Since by Remark \ref{Ap} it holds $%
||\mu _{n}|_{\gamma }||_{W}\leq A^{p-1}(C+M)$ $\forall \gamma ,$ by the
dominated convergence theorem, $||\mu ||_{"1"}\leq C$. Let us now consider
the oscillation. We have that $\underset{n\rightarrow \infty }{\liminf }%
~osc(\mu _{n},x,r)\geq osc(\mu ,x,r)$ for all $x,r$. Indeed, consider a
small $\epsilon \geq 0$. Because of the definition of $osc()$, we have that
for all $l\geq osc(\mu ,x,r)-\epsilon $, there are positive measure sets $%
A_{1}$ and $A_{2}$ such that $W_{1}(\mu |_{\gamma _{1}},\mu |_{\gamma
_{2}})\geq l$ for all $\gamma _{1}\in A_{1},\gamma _{2}\in A_{2}$. Since $%
\mu _{n}|_{\gamma }\rightarrow \mu |_{\gamma }$ for a.e. $\gamma $, if $n$
is big enough there must be sets $A_{1}^{n}$ and $A_{2}^{n}$ of positive
measure, such that $W_{1}(\mu _{n}|_{\gamma _{1}},\mu _{n}|_{\gamma
_{2}})\geq l$ for all $\gamma _{1}\in A_{1}^{n},\gamma _{2}\in A_{2}^{n}$,
by this $osc(\mu _{n},x,r)\geq l$, then for all $x,r,$ $\underset{%
n\rightarrow \infty }{\lim \inf }osc(\mu _{n},x,r)\geq osc(\mu ,x,r)$. By
Fatou's Lemma, $\underset{n\rightarrow \infty }{\lim \inf }var_{p}(\mu
_{n},r)=\underset{n\rightarrow \infty }{\lim \inf }\int_{I}r^{-p}osc(\mu
_{n},x,r)~dx\geq var_{p}(\mu ,r)$. From which the statement follows directly.
\end{proof}

\begin{theorem}[Helly-selection-like theorem]
\label{hL}Let $\mu _{n}$ be a sequence of probability measures on $X$ such\
that $\ ||\mu _{n}||_{p-BV}\leq M$ for some $M\geq 0$. Then there is $\mu $
with $||\mu ||_{p-BV}\leq M$ and subsequence $\mu _{n_{k}}$ such that 
\begin{equation*}
||\mu _{n_{k}}-\mu ||_{"1"}\rightarrow 0.
\end{equation*}
\end{theorem}

\begin{proof}
Let us discretize in the vertical direction. Let us consider a continuous
projection of the space of probability measures on $M$ on a finite
dimensional space $\pi _{y,\delta }:PM(M)\rightarrow U_{\delta }$ \ ($%
U_{\delta }$ is finite dimensional). Suppose $\pi _{y,\delta }$ is such that 
$||\pi _{y,\delta }(\nu )-\nu ||_{W_{1}}\leq C\delta ,$ $\forall \nu \in
PM(M)$ (such projection can be constructed discretizing the space by a
partition of unity made of Lipschitz functions with support on sets whose
diameter is smaller than $\delta $, see \cite{GMN} for example). Let us
consider the natural extension of this projection to the whole $\mathcal{L}%
^{1}(X)$ space $\pi _{\delta }$:$\mathcal{L}^{1}(X)\rightarrow \mathcal{L}%
^{1}(X)$, defined by $\pi _{\delta }(\mu )|_{\gamma }=\pi _{y,\delta }(\mu
|_{\gamma })$.

Let us consider the sequence $\pi _{\delta }(\mu _{n})$. We have $||\pi
_{\delta }(\mu _{n})||_{p-BV}\leq K_{\delta }M$ where $K_{\delta }$ is the
modulus of continuity of $\pi _{\delta }$. Indeed%
\begin{equation*}
\limfunc{esssup}_{\gamma _{2},\gamma _{1}\in B(x,r)}(W_{1}(\pi _{\delta
}(\mu )|_{\gamma _{1}},\pi _{\delta }(\mu )|_{\gamma _{1}}))\leq K_{\delta }%
\limfunc{esssup}_{\gamma _{2},\gamma _{1}\in B(x,r)}(W_{1}(\mu |_{\gamma
_{1}},\mu |_{\gamma _{1}})).
\end{equation*}

Since after projecting we now are in a space of functions with values in a
finite dimensional space, to the sequence $\pi _{\delta }(\mu _{n})$ we can
apply the classical Helly selection theorem and get that there is a limit
measure $\mu _{\delta }$ and a sub sequence $n_{k}$ such that $\pi _{\delta
}(\mu _{n_{k}})\rightarrow \mu _{\delta }$ in $\mathcal{L}^{1}$ and $\pi
_{\delta }(\mu _{n_{k}})|_{\gamma }\rightarrow \mu _{\delta }|_{\gamma }$
almost everywhere. Let us consider a sequence $\delta _{i}\rightarrow 0$ and
select inductively at every step from the previous selected subsequence $\mu
_{l}$ such that $\pi _{\delta _{i-1}}(\mu _{l})\rightarrow \mu _{\delta
_{i-1}}$ a further subsequence $\mu _{l_{k}}$, such that $\pi _{\delta
_{i}}(\mu _{l_{k}})\rightarrow \mu _{\delta _{i}}$ in $\mathcal{L}^{1}$ and
almost everywhere. Since $\forall \gamma $ and $m\leq i,$ $||\pi _{\delta
_{m}}(\mu _{l_{k}}|_{\gamma })-\mu _{l_{k}}|_{\gamma }||_{W_{1}}\leq C\delta
_{m}$, it holds that for different $\delta _{m}$, $\delta _{j}\geq \delta
_{i}$ 
\begin{equation*}
||\pi _{\delta _{i}}(\mu _{l_{k}}|_{\gamma })-\pi _{\delta _{j}}(\mu
_{l_{k}}|_{\gamma })||_{W_{1}}\leq C(\delta _{i}+\delta _{j})
\end{equation*}%
and then $\forall \gamma $%
\begin{equation*}
||\mu _{\delta _{m}}|_{\gamma }-\mu _{\delta _{j}}|_{\gamma }||_{W_{1}}\leq
C(\delta _{m}+\delta _{j}+\delta _{i})
\end{equation*}%
uniformly in $\gamma .$ Since $\mu _{\delta _{n}}$ are positive measures,
this shows that there is a $\mu $ such that \ $\mu _{\delta _{i}}\rightarrow
\mu $ in $\mathcal{L}^{1}$ and $\mu _{\delta _{i}}|_{\gamma }\rightarrow \mu
|_{\gamma }$ almost everywhere. \ This shows that a further subsequence $\mu
_{n_{_{j}}}$ can be selected in a way that $\pi _{\delta _{i}}(\mu
_{n_{_{j}}})\underset{j\rightarrow \infty }{\rightarrow }\mu _{\delta _{i}}$
for all $i$, and $\mu _{n_{_{j}}}\underset{j\rightarrow \infty }{\rightarrow 
}\mu $ in $\mathcal{L}^{1}$and almost everywhere. Applying Lemma \ref{spp}
we get $||\mu ||_{p-BV}\leq M$.
\end{proof}

\section{A regularization inequality \label{LY}}

In this section we prove an inequality, showing that iterates of a bounded
variation positive measure are of uniform bounded variation. This will play
the role of a Lasota Yorke inequality. A consequence will be a bound on the
variation of invariant measures in $\mathcal{L}^{1}$. \ This will be used
when applying Theorem \ref{gen} to provide the estimate \ needed at Item 2
of. The following regularization inequality can be proved.

\begin{proposition}
\label{LYYY}Let $F$ be a skew product map satisfying assumptions $%
Sk1,...,Sk3 $ and let us suppose that $\mu $ is a positive measure. Let $%
p\leq \xi $ (the H\"{o}lder exponent as in $Sk1$). It holds%
\begin{equation*}
var_{p}(L_{F}\mu ){\leq \lambda }^{p}{\alpha ~var}_{p}{(\mu )+}({\hat{H}}%
||\mu _{x}||_{\infty }+3q\alpha C_{h}A^{\xi -p}||\mu _{x}||_{\infty }).
\end{equation*}%
We recall that here $\mu _{x}$ is the marginal of the disintegration of $\mu 
$ (see Equation \ref{eee6}) and $||\mu _{x}||_{\infty }$ is the supremum
norm for its density.
\end{proposition}

\begin{proof}
By the Perron Frobenius like formula (Lemma \ref{PF})%
\begin{equation}
(L_{F}\mu )|_{\gamma }=\sum_{i=1}^{q}{\dfrac{\func{F}_{T_{i}^{-1}(\gamma
)}^{\ast }\mu |_{T_{i}^{-1}(\gamma )}}{|T_{i}^{^{\prime }}\circ
T_{i}^{-1}(\gamma )|}}\ \ \mathnormal{for~almost~all}\ \ \gamma \in \lbrack
0,1]
\end{equation}%
we have to estimate%
\begin{eqnarray*}
I &:&=\sup_{r\leq A}\frac{1}{r^{p}}\int \limfunc{esssup}_{\gamma _{2},\gamma
_{1}\in B(x,r)}||(L_{F}\mu )|_{\gamma _{1}}-(L_{F}\mu )|_{\gamma _{2}}{||}%
_{W_{1}}{~dm(}x) \\
&=&\sup_{r\leq A}\frac{1}{r^{p}}\int \limfunc{esssup}_{\gamma _{2},\gamma
_{1}\in B(x,r)}||\sum_{i=1}^{q}({\dfrac{\func{F}_{T_{i}^{-1}(\gamma
_{1})}^{\ast }\mu |_{T_{i}^{-1}(\gamma _{1})}}{|T_{i}^{^{\prime }}\circ
T_{i}^{-1}(\gamma _{1})|}-\dfrac{\func{F}_{T_{i}^{-1}(\gamma _{2})}^{\ast
}\mu |_{T_{i}^{-1}(\gamma _{2})}}{|T_{i}^{^{\prime }}\circ T_{i}^{-1}(\gamma
_{2})|})||}_{W_{1}}{~dm(}x).
\end{eqnarray*}

To compact notations let us set in next equations%
\begin{equation*}
\func{F}^{\ast }(a,b):=\func{F}_{T_{i}^{-1}(a)}^{\ast }\mu
|_{T_{i}^{-1}(b)},~g_{i}(a):=T_{i}^{^{\prime }}\circ T_{i}^{-1}(a)
\end{equation*}%
\ By the triangular inequality%
\begin{eqnarray*}
I &{\leq }&{\sup_{r\leq A}\sum_{i=1}^{q}\frac{1}{r^{p}}\int \limfunc{esssup}%
_{\gamma _{2},\gamma _{1}\in B(x,r)}||\dfrac{\func{F}^{\ast }(\gamma
_{1},\gamma _{1})-\func{F}^{\ast }(\gamma _{2},\gamma _{2})}{|g_{i}(\gamma
_{1})|}{|}|_{W_{1}}dm} \\
&&+\sup_{r\leq A}\sum_{i=1}^{q}\frac{1}{r^{p}}\int \limfunc{esssup}_{\gamma
_{2},\gamma _{1}\in B(x,r)}||\func{F}^{\ast }(\gamma _{2},\gamma _{2})(\frac{%
1}{|g_{i}(\gamma _{1})|}-\frac{1}{|g_{i}(\gamma _{2})|}){|}|_{W_{1}}dm.
\end{eqnarray*}%
Recalling that $\frac{1}{|g_{i}(\gamma _{2}))|}-\frac{1}{|g_{i}(\gamma
_{1}))|}$ $\leq C_{h}d(\gamma _{1},\gamma _{2})^{\xi }$, then%
\begin{gather*}
I\leq {\sup_{r\leq A}\sum_{i=1}^{q}\frac{1}{r^{p}}}\int (\frac{1}{|g_{i}(x)|}%
+C_{h}r^{\xi }){\limfunc{esssup}_{\gamma _{2},\gamma _{1}\in B(x,r)}||\func{F%
}^{\ast }(\gamma _{1},\gamma _{1})-\func{F}^{\ast }(\gamma _{2},\gamma _{2})|%
}|_{W_{1}}{dm} \\
+\sup_{r\leq A}\sum_{i=1}^{q}\frac{1}{r^{p}}\int C_{h}r^{\xi }\limfunc{esssup%
}_{\gamma _{2}}||\func{F}^{\ast }(\gamma _{2},\gamma _{2}){|}|_{W_{1}}dm.
\end{gather*}

And%
\begin{eqnarray*}
I &{\leq }&{\sup_{r\leq A}\sum_{i=1}^{q}\frac{1}{r^{p}}}\int \frac{1}{%
|g_{i}(x)|}{\limfunc{esssup}_{\gamma _{2},\gamma _{1}\in B(x,r)}||\func{F}%
^{\ast }(\gamma _{1},\gamma _{1})-\func{F}^{\ast }(\gamma _{2},\gamma _{2})|}%
|_{W_{1}}{dm} \\
&&+{\sup_{r\leq A}\sum_{i=1}^{q}\frac{1}{r^{p}}}\int C_{h}r^{\xi }{\limfunc{%
esssup}_{\gamma _{2},\gamma _{1}\in B(x,r)}||\func{F}^{\ast }(\gamma
_{1},\gamma _{1})-\func{F}^{\ast }(\gamma _{2},\gamma _{2})|}|_{W_{1}}{dm} \\
&&+\sup_{r\leq A}\sum_{i=1}^{q}\frac{1}{r^{p}}\int C_{h}r^{\xi }\limfunc{%
esssup}_{\gamma _{2}}||\func{F}^{\ast }(\gamma _{2},\gamma _{2}){|}%
|_{W_{1}}dm.
\end{eqnarray*}

Hence 
\begin{eqnarray*}
I &{\leq }&{\sup_{r\leq A}\sum_{i=1}^{q}\frac{1}{r^{p}}}\int \frac{1}{%
|g_{i}(x)|}{\limfunc{esssup}_{\gamma _{2},\gamma _{1}\in B(x,r)}||\func{F}%
^{\ast }(\gamma _{1},\gamma _{1})-\func{F}^{\ast }(\gamma _{1},\gamma _{2})|}%
|_{W_{1}}{dm} \\
&&+{\sup_{r\leq A}\sum_{i=1}^{q}\frac{1}{r^{p}}}\int \frac{1}{|g_{i}(x)|}{%
\limfunc{esssup}_{\gamma _{2},\gamma _{1}\in B(x,r)}||\func{F}^{\ast
}(\gamma _{1},\gamma _{2})-\func{F}^{\ast }(\gamma _{2},\gamma _{2})|}%
|_{W_{1}}{dm} \\
&&+3\sup_{r\leq A}\sum_{i=1}^{q}\frac{1}{r^{p}}\int C_{h}r^{\xi }\limfunc{%
esssup}_{\gamma _{2}}||\func{F}^{\ast }(\gamma _{2},\gamma _{2}){|}%
|_{W_{1}}dm \\
&=&I_{a}+I_{b}+I_{c}.
\end{eqnarray*}%
Now%
\begin{equation*}
I_{a}{\leq \sup_{r\leq A}\sum_{i=1}^{q}\frac{1}{r^{p}}}\int \frac{1}{%
|g_{i}(x)|}{\limfunc{esssup}_{\gamma _{2},\gamma _{1}\in B(x,r)}||}\func{F}%
_{T_{i}^{-1}(\gamma _{1})}^{\ast }\left( \mu |_{T_{i}^{-1}(\gamma _{1})}-\mu
|_{T_{i}^{-1}(\gamma _{2})}\right) {|}|_{W_{1}}.
\end{equation*}

We recall that by Lemma \ref{unamis} $\ ||\func{F}_{\gamma }^{\ast }\mu
||_{W_{1}}\leq \alpha ||\mu ||_{W_{1}}$ then%
\begin{eqnarray*}
I_{a} &{\leq }&{\sup_{r\leq A}\sum_{i=1}^{q}\frac{1}{r^{p}}}\int \frac{1}{%
|g_{i}(x)|}{\limfunc{esssup}_{\gamma _{2},\gamma _{1}\in B(x,r)}}\alpha
||\mu |_{T_{i}^{-1}(\gamma _{1})}-\mu |_{T_{i}^{-1}(\gamma _{2})}{|}|_{W_{1}}
\\
&\leq &{\sup_{r\leq A}\sum_{i=1}^{q}\frac{1}{r^{p}}}\int_{I_{i}}{\limfunc{%
esssup}_{y_{1},y_{2}\in B(x,\lambda r)}}\alpha ||\mu |_{y_{1}}-\mu |_{y_{2}}{%
|}|_{W_{1}}dx. \\
&=&\lambda ^{p}{\sup_{h\leq \lambda A}\sum_{i=1}^{q}\frac{1}{h^{p}}}%
\int_{I_{i}}{\limfunc{esssup}_{y_{1},y_{2}\in B(x,h)}}\alpha ||\mu
|_{y_{1}}-\mu |_{y_{2}}{|}|_{W_{1}}dx
\end{eqnarray*}

and%
\begin{equation*}
var_{p}(L_{F}\mu ){\leq \lambda }^{p}{\alpha ~var}_{p}{(\mu )+}(I_{b}+I_{c}).
\end{equation*}

By $Sk3$%
\begin{eqnarray*}
I_{b} &=&{\sup_{r\leq A}\sum_{i}\frac{1}{r^{p}}}\int \frac{1}{|g_{i}(x)|}{%
\limfunc{esssup}_{y_{1},y_{2}\in B(x,r)}||\func{F}^{\ast }(\gamma
_{1},\gamma _{2})-\func{F}^{\ast }(\gamma _{2},\gamma _{2})|}|_{W_{1}}{dm} \\
&{\leq }&{\hat{H}}||\mu _{x}||_{\infty }.
\end{eqnarray*}%
Now, let us remark that since we are working with positive measures 
\begin{equation*}
\limfunc{esssup}_{\gamma _{2}}||\func{F}^{\ast }(\gamma _{2},\gamma _{2}){|}%
|_{W_{1}}\leq \alpha ||\mu _{x}||_{\infty },
\end{equation*}%
then%
\begin{eqnarray}
I_{c} &=&3\sup_{r\leq A}\sum_{i=1}^{q}\frac{1}{r^{p}}\int C_{h}r^{\xi }%
\limfunc{esssup}_{\gamma _{2}}||\func{F}^{\ast }(\gamma _{2},\gamma _{2}){|}%
|_{W_{1}}dm  \notag \\
&\leq &3\sup_{r\leq A}\sum_{i=1}^{q}\frac{1}{r^{p}}\int C_{h}r^{\xi }\alpha
||\mu _{x}||_{\infty }dm \\
&\leq &3qC_{h}A^{\xi -p}\alpha ||\mu _{x}||_{\infty }.
\end{eqnarray}

Summarizing%
\begin{equation}
var_{p}(L_{F}\mu ){\leq \lambda }^{p}{\alpha ~var}_{p}{(\mu )+}({\hat{H}}%
||\mu _{x}||_{\infty }+3q\alpha C_{h}A^{\xi -p}||\mu _{x}||_{\infty }).
\label{lxt}
\end{equation}
\end{proof}

\begin{remark}
By Equation \ref{ly1d} it holds that for each $n$%
\begin{equation*}
||L^{n}\mu _{x}||_{\infty }\leq C_{\mu }:=A^{p-1}(A_{T}\lambda ||\mu
_{x}||_{BV}+B_{T}||\mu _{x}||_{1}+1).
\end{equation*}
Iterating (\ref{lxt}) we obtain%
\begin{eqnarray}
var_{p}(L_{F}^{n}\mu ) &{\leq }&{(\lambda }^{p}{\alpha )~var}_{p}{(L}^{n-1}{%
\mu )+}({\hat{H}}+3q\alpha C_{h}A^{\xi -p})C_{\mu }  \label{iter} \\
&\leq &...  \notag \\
&\leq &{(\lambda }^{p}{\alpha )}^{n}{~var}_{p}{(\mu )+}\frac{{\hat{H}}%
+3q\alpha C_{h}A^{\xi -p}}{1-{\lambda }^{p}{\alpha }}C_{\mu }  \notag
\end{eqnarray}
\end{remark}

By the Helly-like selection principle (Theorem \ref{hL}) we then have

\begin{proposition}
\label{regu}In a system as above there is at least an invariant positive
measure in $p-BV$. For every such invariant measure $\mu $ 
\begin{equation*}
var_{p}(\mu )\leq \frac{B_{T}({\hat{H}}+3q\alpha C_{h}A^{\xi -p})}{1-{%
\lambda }^{p}{\alpha }}||\mu ||_{"1"}.
\end{equation*}
\end{proposition}

\begin{proof}
We consider the sequence of positive measures $\mu _{n}=\frac{1}{n}\sum
L_{F}^{n}m$. By Equation \ref{iter}, this sequence has uniformly bounded
variation. Applying Theorem \ref{hL}, we deduce the existence of an
invariant measure $\mu $ in $p-BV$. By the Lasota Yorke inequality relative
to the map $T$, we have that 
\begin{equation*}
A^{p-1}B_{T}||\mu ||_{"1"}\geq A^{p-1}B_{T}||\mu _{x}||_{1}\geq A^{p-1}||\mu
_{x}||_{BV}\geq ||\mu _{x}||_{\infty }.
\end{equation*}%
This gives that 
\begin{equation*}
var_{p}(\mu )=var_{p}(L_{F}\mu ){\leq \lambda }^{p}{\alpha ~var}_{p}{(\mu )+}%
B_{T}||\mu ||_{"1"}A^{p-1}({\hat{H}}+3\alpha qC_{h}A^{\xi -p})
\end{equation*}%
from which we get the statement.
\end{proof}

\section{Distance between the operators and a general statement for skew
products\label{Dist}}

Here we consider a suitable class of perturbations of a map satisfying $%
Sk1,...,Sk3$ such that the associated transfer operators are near in the
strong-weak topology, providing one of the estimates needed to apply Theorem %
\ref{gen} (item 4). In this section and in the following we set $p=1$. Now
we define a topology on the space of piecewise expanding maps to have a
notion of "allowed perturbations" for these maps.

\begin{definition}
Let $T_{1}$ and $T_{2}$ be to piecewise expanding maps. Denote by%
\begin{equation*}
Int_{n}=\{A\in 2^{[0,1]},s.t.\ A=I_{1}\cup ,...,\cup I_{n},\ \mathnormal{%
where}\ I_{i}\ \mathnormal{are~intervals}\}
\end{equation*}%
the set of subsets of $[0,1]$ which is the union of at most $n$ intervals.
Set 
\begin{equation*}
\mathcal{C}(n,T_{1},T_{2})=\left\{ 
\begin{array}{c}
\epsilon :\exists A_{1}\in Int_{n}\ \mathnormal{and}\ \exists \ \sigma
:I\rightarrow I\ \mathnormal{a~diffeomorphism}\ \mathnormal{s.t.} \\ 
m(A_{1})\geq 1-\epsilon ,\ T_{1}|_{A_{1}}=T_{2}\circ \sigma |_{A_{1}} \\ 
\mathnormal{and}\ \forall x\in A_{1},|\sigma (x)-x|\leq \epsilon ,|\frac{1}{%
\sigma ^{\prime }(x)}-1|\leq \epsilon%
\end{array}%
\right\}
\end{equation*}%
and define a kind of distance from $T_{1}$ to $T_{2}$ as:%
\begin{equation}
d_{S,n}(T_{1},T_{2})=\inf \left\{ \epsilon |\epsilon \in \mathcal{C}%
(n,T_{1},T_{2})\right\} .  \label{shock}
\end{equation}
\end{definition}

It holds that one dimensional piecewise expanding maps which are near in the
sense of $d_{S,n}$ also have transfer operators which are near as operators
from $BV$ to $L^{1}$. If we denote by $d_{S}$ the classical notion of
Skorokhod distance (see \cite{BG} e.g.), it is obvious that $\forall n\
d_{S,n}\geq d_{S}$. By \cite{BG}, Lemma 11.2.1, it follows that $\forall n$
there is $C_{Sk}\geq 0$ such that for each pair of piecewise expanding maps $%
T_{1},T_{2}$%
\begin{equation}
||L_{T_{0}}-L_{T_{\delta }}||_{BV\rightarrow L^{1}}\leq
C_{Sk}d_{n,S}(T_{1},T_{2}).  \label{gura}
\end{equation}

Let us see a statement of this kind for our skew products.

\begin{proposition}
\label{UF}Let $F_{\delta }=(T_{\delta },G_{\delta }),$ $0\leq \delta \leq D$
be a family of maps satisfying $Sk1,...,Sk3$ uniformly with $\xi =1$ and:

\begin{enumerate}
\item There is $n\in \mathbb{N}$ such that for each $\delta \leq
D~,~d_{n,S}(T_{0},T_{\delta })\leq \delta .$ (thus for each $\delta $ there
is a set $A_{1}\in Int_{n}$ as in the definition of $\mathcal{C}%
(n,T_{1},T_{2})$)

\item For each $\delta \leq D$ there is a set $A_{2}\in Int_{n}$ such that $%
m(A_{2})\geq 1-\delta $ and for all $x\in A_{2},y\in M:$ $%
|G_{0}(x,y)-G_{\delta }(x,y)|\leq \delta .$
\end{enumerate}

Let us denote by $\func{F}_{\delta }^{\ast }$ the transfer operators of $%
F_{\delta }$ and by $f_{\delta }$ a family of probability measures with
uniformly bounded variation 
\begin{equation*}
var_{1}(f_{\delta })\leq M_{2}.
\end{equation*}%
Then, there is a constants $C_{1}$ such that for $\delta $ small enough%
\begin{equation*}
||(\func{F}_{0}^{\ast }-\func{F}_{\delta }^{\ast })f_{\delta }||_{"{1"}}\leq
C_{1}\delta (M_{2}+1).
\end{equation*}
\end{proposition}

\begin{proof}
Let us set $A=A_{1}\cap A_{2}$. Note that $m(A^{c})\leq 2\delta .$ Let us
estimate%
\begin{eqnarray}
||(\func{F}_{0}^{\ast }-\func{F}_{\delta }^{\ast })f_{\delta }||_{{1}}
&=&\int_{I}{||({F_{0}^{\ast }}f_{\delta }-{F_{\delta }^{\ast }}f_{\delta
})|_{\gamma }||_{W}}dm(\gamma )  \label{12112} \\
&=&\int_{I}{||{F_{0}^{\ast }(1}}_{A}{f_{\delta })|}_{\gamma }{-{F_{\delta
}^{\ast }({1}}_{A}f_{\delta })|_{\gamma }||_{W}}dm(\gamma )  \notag \\
&&+\int_{I}{||{F_{0}^{\ast }(1}}_{A^{c}}{f_{\delta })|}_{\gamma }{-{%
F_{\delta }^{\ast }({1}}_{A^{c}}f_{\delta })|_{\gamma }||_{W}}dm(\gamma ).
\end{eqnarray}

By the assumptions, for a.e. $\gamma ,$ $||{f_{\delta }|\gamma ||}_{W}\leq
(M_{2}+1)$ and $||{1}_{A^{c}}{f_{\delta }||}_{1}\leq (M_{2}+1)\delta .$
Since ${F^{\ast }}$ is $\alpha $-Lipschitz for the $\mathcal{L}^{1}$norm
then 
\begin{equation*}
\int_{I}{||{F_{0}^{\ast }(1}}_{A^{c}}{f_{\delta })|}_{\gamma }{-{F_{\delta
}^{\ast }({1}}_{A^{c}}f_{\delta })|_{\gamma }||_{W}}dm(\gamma )\leq 2\alpha
(M_{2}+1)\delta .
\end{equation*}

Let us now estimate the first summand of \ref{12112}. Let us set $\mu
=1_{A}f_{\delta }$ and let us estimate%
\begin{equation*}
||(\func{F}_{0}^{\ast }-\func{F}_{\delta }^{\ast })\mu ||_{{1}}=\int {||({%
F_{0}^{\ast }}\mu -{F_{\delta }^{\ast }}\mu )|_{\gamma }||_{W}}dm(\gamma ).
\end{equation*}%
Let us denote by $T_{0,i}$, with $0\leq i\leq q$ the branches of $T_{0}$
defined in the sets $P_{i}$, partition of $I,$ and set $T_{\delta
,i}=T_{\delta }|_{P_{i}\cap A}$ these functions will play the role of the
branches for $T_{\delta }.$ Since in $A,$ $T_{0}=T_{\delta }\circ \sigma
_{\delta }$ (where $\sigma _{\delta }$ is the diffeomorphism in the
definition of the Skorokhod distance), then\ $T_{\delta ,i}$ are invertible.
Then for $\mu _{x}-a.e.\ \gamma \in I$%
\begin{equation*}
({F_{0}^{\ast }}\mu -{F_{\delta }^{\ast }}\mu )|_{\gamma }=\sum_{i=1}^{q}%
\frac{{F_{0,T_{0,i}^{-1}(\gamma )}^{\ast }}\mu |_{T_{0,i}^{-1}(\gamma )}\chi
_{T_{0}(P_{i}\cap A)}}{|T_{0,i}^{\prime }(T_{0,i}^{-1}(\gamma ))|}%
-\sum_{i=1}^{q}\frac{{F_{\delta ,T_{\delta ,i}^{-1}(\gamma )}^{\ast }}\mu
|_{T_{\delta ,i}^{-1}(\gamma )}\chi _{T_{\delta }(P_{i}\cap A)}}{|T_{\delta
,i}^{\prime }(T_{\delta ,i}^{-1}(\gamma ))|}.
\end{equation*}%
Let us now consider $T_{0}(P_{i}\cap A)$ and $T_{\delta }(P_{i}\cap A),$ and
remark that $T_{0}(P_{i}\cap A)=\sigma _{\delta }(T_{\delta }(P_{i}\cap A))$
\ and $\sigma _{\delta }$ is a diffeomorphism near to the identity. Let us
denote $B_{i}=T_{0}(P_{i}\cap A)\cap T_{\delta }(P_{i}\cap A),\
C_{i}=T_{0}(P_{i}\cap A)\triangle T_{\delta }(P_{i}\cap A)$. 
\begin{eqnarray}
||(\func{F}_{0}^{\ast }-\func{F}_{\delta }^{\ast })\mu ||_{{1}} &=&\int_{I}{%
||({F_{0}^{\ast }}\mu -{F_{\delta }^{\ast }}\mu )|_{\gamma }||_{W}}dm(\gamma
)  \label{23231} \\
&\leq &\int_{I}\left\vert \left\vert \sum_{i=1}^{q}\frac{{%
F_{0,T_{0,i}^{-1}(\gamma )}^{\ast }}\mu |_{T_{0,i}^{-1}(\gamma )}\chi
_{B_{i}}}{|T_{0,i}^{\prime }(T_{0,i}^{-1}(\gamma ))|}-\sum_{i=1}^{q}\frac{{%
F_{\delta ,T_{\delta ,i}^{-1}(\gamma )}^{\ast }}\mu |_{T_{\delta
,i}^{-1}(\gamma )}\chi _{B_{i}}}{|T_{\delta ,i}^{\prime }(T_{\delta
,i}^{-1}(\gamma ))|}\right\vert \right\vert _{W}dm  \notag \\
&&+\int_{I}\left\vert \left\vert \sum_{i=1}^{q}\frac{{F_{0,T_{0,i}^{-1}(%
\gamma )}^{\ast }}\mu |_{T_{0,i}^{-1}(\gamma )}\chi _{C_{i}}}{%
|T_{0,i}^{\prime }(T_{0,i}^{-1}(\gamma ))|}-\sum_{i=1}^{q}\frac{{F_{\delta
,T_{\delta ,i}^{-1}(\gamma )}^{\ast }}\mu |_{T_{\delta ,i}^{-1}(\gamma
)}\chi _{C_{i}}}{|T_{\delta ,i}^{\prime }(T_{\delta ,i}^{-1}(\gamma ))|}%
\right\vert \right\vert _{W}dm.  \notag
\end{eqnarray}

And since there is $K_{1}$ such that $m(C_{i})\leq K_{1}\delta $, then%
\begin{equation*}
\int_{I}\left\vert \left\vert \sum_{i=1}^{q}\frac{{F_{0,T_{0,i}^{-1}(\gamma
)}^{\ast }}\mu |_{T_{0,i}^{-1}(\gamma )}\chi _{C_{i}}}{|T_{0,i}^{\prime
}(T_{0,i}^{-1}(\gamma ))|}-\sum_{i=1}^{q}\frac{{F_{\delta ,T_{\delta
,i}^{-1}(\gamma )}^{\ast }}\mu |_{T_{\delta ,i}^{-1}(\gamma )}\chi _{C_{i}}}{%
|T_{\delta ,i}^{\prime }(T_{\delta ,i}^{-1}(\gamma ))|}\right\vert
\right\vert _{W}dm\leq qK_{1}(M_{2}+1)\delta .
\end{equation*}

Now we have to consider the first summand of \ref{23231}. We have%
\begin{eqnarray*}
&&\int_{I}\left\vert \left\vert \sum_{i=1}^{q}\frac{{F_{0,T_{0,i}^{-1}(%
\gamma )}^{\ast }}\mu |_{T_{0,i}^{-1}(\gamma )}\chi _{B_{i}}}{%
|T_{0,i}^{\prime }(T_{0,i}^{-1}(\gamma ))|}-\sum_{i=1}^{q}\frac{{F_{\delta
,T_{\delta ,i}^{-1}(\gamma )}^{\ast }}\mu |_{T_{\delta ,i}^{-1}(\gamma
)}\chi _{B_{i}}}{|T_{\delta ,i}^{\prime }(T_{\delta ,i}^{-1}(\gamma ))|}%
\right\vert \right\vert _{W}dm \\
&\leq &\int_{I}\left\vert \left\vert \sum_{i=1}^{q}\frac{{%
F_{0,T_{0,i}^{-1}(\gamma )}^{\ast }}\mu |_{T_{0,i}^{-1}(\gamma )}\chi
_{B_{i}}}{|T_{0,i}^{\prime }(T_{0,i}^{-1}(\gamma ))|}-\sum_{i=1}^{q}\frac{{%
F_{\delta ,T_{\delta ,i}^{-1}(\gamma )}^{\ast }}\mu |_{T_{0,i}^{-1}(\gamma
)}\chi _{B_{i}}}{|T_{\delta ,i}^{\prime }(T_{\delta ,i}^{-1}(\gamma ))|}%
\right\vert \right\vert _{W}dm \\
&&+\int_{I}\left\vert \left\vert \sum_{i=1}^{q}\frac{{F_{\delta ,T_{\delta
,i}^{-1}(\gamma )}^{\ast }}\mu |_{T_{0,i}^{-1}(\gamma )}\chi _{B_{i}}}{%
|T_{\delta ,i}^{\prime }(T_{\delta ,i}^{-1}(\gamma ))|}-\sum_{i=1}^{q}\frac{{%
F_{\delta ,T_{\delta ,i}^{-1}(\gamma )}^{\ast }}\mu |_{T_{\delta
,i}^{-1}(\gamma )}\chi _{B_{i}}}{|T_{\delta ,i}^{\prime }(T_{\delta
,i}^{-1}(\gamma ))|}\right\vert \right\vert _{W}dm \\
&=&\int_{I}I(\gamma )~dm(\gamma )+\int_{I}II(\gamma )~dm(\gamma ).
\end{eqnarray*}

The two summands will be treated separately.%
\begin{eqnarray*}
I(\gamma ) &=&\left\vert \left\vert \sum_{i=1}^{q}\frac{{F_{0,T_{0,i}^{-1}(%
\gamma )}^{\ast }}\mu |_{T_{0,i}^{-1}(\gamma )}\chi _{B_{i}}}{%
|T_{0,i}^{\prime }(T_{0,i}^{-1}(\gamma ))|}-\sum_{i=1}^{q}\frac{{F_{\delta
,T_{\delta ,i}^{-1}(\gamma )}^{\ast }}\mu |_{T_{0,i}^{-1}(\gamma )}\chi
_{B_{i}}}{|T_{\delta ,i}^{\prime }(T_{\delta ,i}^{-1}(\gamma ))|}\right\vert
\right\vert _{W} \\
&\leq &\left\vert \left\vert \sum_{i=1}^{q}\frac{{F_{0,T_{0,i}^{-1}(\gamma
)}^{\ast }}\mu |_{T_{0,i}^{-1}(\gamma )}\chi _{B_{i}}}{|T_{0,i}^{\prime
}(T_{0,i}^{-1}(\gamma ))|}-\sum_{i=1}^{q}\frac{{F_{\delta ,T_{\delta
,i}^{-1}(\gamma )}^{\ast }}\mu |_{T_{0,i}^{-1}(\gamma )}\chi _{B_{i}}}{%
|T_{0,i}^{\prime }(T_{0,i}^{-1}(\gamma ))|}\right\vert \right\vert _{W} \\
&+&\left\vert \left\vert \sum_{i=1}^{q}\frac{{F_{\delta ,T_{\delta
,i}^{-1}(\gamma )}^{\ast }}\mu |_{T_{0,i}^{-1}(\gamma )}\chi _{B_{i}}}{%
|T_{0,i}^{\prime }(T_{0,i}^{-1}(\gamma ))|}-\sum_{i=1}^{q}\frac{{F_{\delta
,T_{\delta ,i}^{-1}(\gamma )}^{\ast }}\mu |_{T_{0,i}^{-1}(\gamma )}\chi
_{B_{i}}}{|T_{\delta ,i}^{\prime }(T_{\delta ,i}^{-1}(\gamma ))|}\right\vert
\right\vert _{W} \\
&=&I_{a}(\gamma )+I_{b}(\gamma ).
\end{eqnarray*}%
Since $f_{\delta }$ is a probability measure it holds posing $\beta
=T_{0,i}^{-1}(\gamma )$%
\begin{eqnarray*}
\int I_{a}(\gamma )dm &=&\int \left\vert \left\vert \sum_{i=1}^{q}\frac{{%
F_{0,T_{0,i}^{-1}(\gamma )}^{\ast }}\mu |_{T_{0,i}^{-1}(\gamma )}\chi
_{B_{i}}}{|T_{0,i}^{\prime }(T_{0,i}^{-1}(\gamma ))|}-\sum_{i=1}^{q}\frac{{%
F_{\delta ,T_{\delta ,i}^{-1}(\gamma )}^{\ast }}\mu |_{T_{0,i}^{-1}(\gamma
)}\chi _{B_{i}}}{|T_{0,i}^{\prime }(T_{0,i}^{-1}(\gamma ))|}\right\vert
\right\vert _{W}dm(\gamma ) \\
&\leq &\int \sum_{i=1}^{q}\left\vert \left\vert \frac{{F_{0,T_{0,i}^{-1}(%
\gamma )}^{\ast }}\mu |_{T_{0,i}^{-1}(\gamma )}\chi _{B_{i}}}{%
|T_{0,i}^{\prime }(T_{0,i}^{-1}(\gamma ))|}-\frac{{F_{\delta ,T_{\delta
,i}^{-1}(\gamma )}^{\ast }}\mu |_{T_{0,i}^{-1}(\gamma )}\chi _{B_{i}}}{%
|T_{0,i}^{\prime }(T_{0,i}^{-1}(\gamma ))|}\right\vert \right\vert _{W}dm \\
&\leq &\sum_{i=1}^{q}\int \left\vert \left\vert \frac{{F_{0,T_{0,i}^{-1}(%
\gamma )}^{\ast }}\mu |_{T_{0,i}^{-1}(\gamma )}\chi _{B_{i}}}{%
|T_{0,i}^{\prime }(T_{0,i}^{-1}(\gamma ))|}-\frac{{F_{\delta ,T_{\delta
,i}^{-1}(\gamma )}^{\ast }}\mu |_{T_{0,i}^{-1}(\gamma )}\chi _{B_{i}}}{%
|T_{0,i}^{\prime }(T_{0,i}^{-1}(\gamma ))|}\right\vert \right\vert _{W}dm \\
&\leq &\sum_{i=1}^{q}\int_{T_{0,i}^{-1}(B_{i})}\left\vert \left\vert {%
F_{0,\beta }^{\ast }}\mu |_{\beta }-{F_{\delta ,T_{\delta
,i}^{-1}(T_{0,i}(\beta ))}^{\ast }}\mu |_{\beta }\right\vert \right\vert
_{W}dm(\beta )
\end{eqnarray*}

Remark that $T_{0,i}^{-1}(B_{i})\subseteq P_{i}\cap A$ and $T_{\delta
,i}^{-1}(T_{0,i}(T_{0,i}^{-1}(B_{i})))\subseteq P_{i}\cap A$. Since $%
|T_{\delta ,i}(\beta )-T_{0,i}(\beta )|\leq \delta $ and $T_{0,i}^{-1}$ \ is
a contraction, then \ $|T_{0,i}^{-1}\circ T_{\delta ,i}(\beta )-\beta |\leq
\delta $. Then%
\begin{eqnarray*}
\left\vert \left\vert {F_{0,\beta }^{\ast }}\mu |_{\beta }-{F_{\delta
,T_{\delta ,i}^{-1}(T_{0,i}(\beta ))}^{\ast }}\mu |_{\beta }\right\vert
\right\vert _{W} &\leq &\left\vert \left\vert {F_{0,\beta }^{\ast }}\mu
|_{\beta }-{F_{\delta ,\beta }^{\ast }}\mu |_{\beta }\right\vert \right\vert
_{W} \\
&&+\left\vert \left\vert {F_{\delta ,\beta }^{\ast }}\mu |_{\beta }-{%
F_{\delta ,T_{\delta ,i}^{-1}(T_{0,i}(\beta ))}^{\ast }}\mu |_{\beta
}\right\vert \right\vert _{W}.
\end{eqnarray*}%
By assumption (2),%
\begin{equation*}
\left\vert \left\vert {F_{0,\beta }^{\ast }}\mu |_{\beta }-{F_{\delta ,\beta
}^{\ast }}\mu |_{\beta }\right\vert \right\vert _{W}\leq \delta (M_{2}+1).
\end{equation*}%
By assumption $Sk3$ 
\begin{equation*}
\left\vert \left\vert {F_{\delta ,\beta }^{\ast }}\mu |_{\beta }-{F_{\delta
,T_{\delta ,i}^{-1}(T_{0,i}(\beta ))}^{\ast }}\mu |_{\beta }\right\vert
\right\vert _{W}\leq \sup_{y\in M,x_{1},x_{2}\in B(\beta ,\delta
)}|G(x_{1},y)-G(x_{2},y)|(M_{2}+1).
\end{equation*}%
Thus, 
\begin{equation*}
I_{a}(\gamma )\leq \delta ^{p}(\hat{H}+1)(M_{2}+1)+\delta (M_{2}+1).
\end{equation*}%
To estimate $I_{b}(\gamma )$ we have:%
\begin{eqnarray*}
I_{b}(\gamma ) &=&\left\vert \left\vert \sum_{i=1}^{q}\frac{{F_{\delta
,T_{\delta ,i}^{-1}(\gamma )}^{\ast }}\mu |_{T_{0,i}^{-1}(\gamma )}\chi
_{B_{i}}}{|T_{0,i}^{\prime }(T_{0,i}^{-1}(\gamma ))|}-\sum_{i=1}^{q}\frac{{%
F_{\delta ,T_{\delta ,i}^{-1}(\gamma )}^{\ast }}\mu |_{T_{0,i}^{-1}(\gamma
)}\chi _{B_{i}}}{|T_{\delta ,i}^{\prime }(T_{\delta ,i}^{-1}(\gamma ))|}%
\right\vert \right\vert _{W} \\
&\leq &\sum_{i=1}^{q}\left\vert \frac{\chi _{B_{i}}(\gamma )}{%
|T_{0,i}^{\prime }(T_{0,i}^{-1}(\gamma ))|}-\frac{\chi _{B_{i}}(\gamma )}{%
|T_{\delta ,i}^{\prime }(T_{\delta ,i}^{-1}(\gamma ))|}\right\vert
\left\vert \left\vert F_{\delta ,T_{\delta ,i}^{-1}(\gamma )}^{\ast }\mu
|_{T_{0,i}^{-1}(\gamma )}\right\vert \right\vert _{W}
\end{eqnarray*}%
and%
\begin{equation*}
\int I_{b}~dm\leq |({P}_{T_{0}}-{P}_{T_{\delta }}\left) (1)\right\vert
\alpha (M_{2}+1)+qK_{1}\delta .
\end{equation*}%
by Equation \ref{gura} then%
\begin{equation*}
\int_{A_{1}}I_{b}(\gamma )~dm(\gamma )\leq \lbrack C_{Sk}\alpha
(M_{2}+1)+qK_{1}]\delta .
\end{equation*}%
Now, let us estimate the integral of the second summand%
\begin{equation*}
II(\gamma )=\left\vert \left\vert \sum_{i=1}^{q}\frac{{F_{\delta ,T_{\delta
,i}^{-1}(\gamma )}^{\ast }}\mu |_{T_{0,i}^{-1}(\gamma )}\chi _{B_{i}}}{%
|T_{\delta ,i}^{\prime }(T_{\delta ,i}^{-1}(\gamma ))|}-\sum_{i=1}^{q}\frac{{%
F_{\delta ,T_{\delta ,i}^{-1}(\gamma )}^{\ast }}\mu |_{T_{\delta
,i}^{-1}(\gamma )}\chi _{B_{i}}}{|T_{\delta ,i}^{\prime }(T_{\delta
,i}^{-1}(\gamma ))|}\right\vert \right\vert _{W}.
\end{equation*}

Then, setting $g_{i}(a):=T_{\delta ,i}^{^{\prime }}\circ T_{\delta
,i}^{-1}(a)$ to compact notations%
\begin{eqnarray*}
\int_{I}II(\gamma )~dm(\gamma ) &\leq &\sum_{i=1}^{q}\int_{B_{i}}\frac{1}{%
|g_{i}(\gamma )|}\left\vert \left\vert {F_{\delta ,T_{\delta ,i}^{-1}(\gamma
)}^{\ast }}\left( \mu |_{T_{0,i}^{-1}(\gamma )}-\mu |_{T_{\delta
,i}^{-1}(\gamma )}\right) \right\vert \right\vert _{W}dm(\gamma ) \\
&\leq &\sum_{i=1}^{q}\int_{B_{i}}\frac{\alpha }{|g_{i}(\gamma )|}\left\vert
\left\vert \mu |_{T_{0,i}^{-1}(\gamma )}-\mu |_{T_{\delta ,i}^{-1}(\gamma
)}\right\vert \right\vert _{W}dm(\gamma )
\end{eqnarray*}%
Let us consider the change of variable $\gamma =T_{\delta ,i}(\beta ),$then%
\begin{equation*}
\int_{I}II(\gamma )~dm(\gamma )\leq \alpha \sum_{i=1}^{q}\int_{T_{\delta
,i}^{-1}(B_{i})}\left\vert \left\vert \mu |_{T_{0,i}^{-1}(T_{\delta
,i}(\beta ))}-\mu |_{\beta }\right\vert \right\vert _{W}dm(\beta ).
\end{equation*}

Since $|T_{\delta ,i}(\beta )-T_{0,i}(\beta )|\leq \delta $ and $%
T_{0,i}^{-1} $ \ is a contraction, then \ $|T_{0,i}^{-1}\circ T_{\delta
,i}(\beta )-\beta |\leq \delta $ then%
\begin{equation*}
\int_{I}II(\gamma )~dm(\gamma )\leq \alpha \int \sup_{x,y\in B(\beta ,\delta
)}(||\mu |_{x}-\mu |_{y}||_{W})dm(\beta )\leq \alpha \int osc(\mu ,\beta
,\delta )~d\beta
\end{equation*}%
and then 
\begin{equation*}
\int_{I}II(\gamma )~dm(\gamma )\leq \alpha 2\delta (M_{2}+1).
\end{equation*}%
Summing all, the statement is proved.
\end{proof}

The last statement, together with the results of the previous sections
allows to prove the following quantitative statement for skew product maps.

\begin{proposition}
\label{thm}Consider a family of skew product maps $F_{\delta }=(T_{\delta
},G_{\delta }),$ $0\leq \delta \leq D$ satisfying $Sk1,...,Sk3$ uniformly,
with $\xi =1$, and let $f_{\delta }\in \mathcal{L}^{1}$ invariant
probability measures of $F_{\delta },$ suppose:

\begin{enumerate}
\item There is $\phi \in C^{0}(\mathbb{R)},~\phi (t)$ decreasing to $0$ as $%
t\rightarrow \infty $ such that $L_{F_{0}}$ has convergence to equilibrium
with respect to norms $||~||_{1-BV}$, $||~||_{"1"}$ and speed $\phi $;

\item there is $\tilde{C}\geq 0$ such that for each $n$ 
\begin{equation*}
||L_{F_{0}}^{n}||_{\mathcal{L}^{1}\rightarrow \mathcal{L}^{1}}\leq \tilde{C};
\end{equation*}

\item there is $n\in \mathbb{N}$ such that for each $\delta \leq
D~,~d_{n,S}(T_{0},T_{\delta })\leq \delta ;$

\item for each $\delta \leq D$ there is a set $A_{2}\in Int_{n}$ such that $%
m(A_{2})\geq 1-\delta $ and for all $x\in A_{2},y\in M:$ $%
|G_{0}(x,y)-G_{\delta }(x,y)|\leq \delta .$
\end{enumerate}

Let $B=\frac{B_{T}({\hat{H}}+3q\alpha C_{h})}{1-{\lambda }^{p}{\alpha }}+1.$
Consider the function $\psi $ defined as $\psi (x)=\frac{\phi (x)}{x},$ then%
\begin{equation*}
||f_{\delta }-f_{0}||_{"1"}\leq 2\tilde{C}B^{2}C_{1}\delta (\psi ^{-1}(\frac{%
\tilde{C}BC_{1}\delta }{2})+1).
\end{equation*}%
where $C_{1}$ is the constant in the statement of Proposition \ref{UF}.
\end{proposition}

\begin{proof}
The proof is a direct application of the estimates given in the previous
section into Theorem \ref{gen}. The quantity $\tilde{M}$ appearing at Item 2
of Theorem \ref{gen} is estimated by Proposition \ref{regu}:%
\begin{equation*}
\tilde{M}\leq \frac{B_{T}({\hat{H}}+3q\alpha C_{h})}{1-{\lambda }^{p}{\alpha 
}}.
\end{equation*}%
By Proposition \ref{UF} the distance between the operators appearing at Item
4 of Theorem \ref{gen} is bounded by $\epsilon \leq C_{1}\delta (M_{2}+1)$
Where $M_{2}$ bounds the strong norm of $f_{\delta }$.
\end{proof}

We remark that the quantitative stability is proved here in the $||~||_{"1"}$
topology. This topology is strong enough to control the behavior of
observables which are discontinuous along the preserved central foliation,
see \cite{BKL} for other results on quantitative stability of the
statistical properties of discontinuous observables and related applications.

In the following section we show a class of nontrivial partially hyperbolic
skew products having power law convergence to equilibrium and will apply
this statement to these examples.

\section{Application to slowly mixing toral extensions \label{Sec2}}

To give an example of application of Proposition \ref{thm} to a class of
nontrivial system, we consider a class of "partially hyperbolic" skew
products with some discontinuities, having slow (power law) decay of
correlations and convergence to equilibrium.

We will consider a class of skew products $F=(T,G)$ (piecewise constant
toral extensions) defined as follows:

\begin{enumerate}
\item[Te1] let $l\in \mathbb{N}$. We assume that $T$ is the piecewise
expanding map on $[0,1]$ defined as 
\begin{equation*}
T(x)=lx~\func{mod}(1);
\end{equation*}

\item[Te2] the system is extended by a skew product to a system $(X,F)$
where $X=[0,1]\times \mathcal{T}^{d}$, where $\mathcal{T}^{d}$ is the $d$
dimensional torus and $F:X\rightarrow X$ is defined by%
\begin{equation}
F(x,t)=(Tx,t+\theta \varphi (x))  \label{skewprod}
\end{equation}%
where $\theta =(\theta _{1},...,\theta _{d})\in \mathcal{T}^{d}$ and $%
\varphi =1_{I}$ is the characteristic function of a set $I\subset \lbrack
0,1]$ which is an union of the sets $P_{i}$ where the branches of $T$ are
defined. In this system the second coordinate is translated by $\theta $ if
the first coordinate belongs to $I$.
\end{enumerate}

We remark that on the system $(X,F)$ the Lebesgue measure is invariant. We
will suppose that $\theta $ is of finite Diophantine type. Let us recall the
definition of Diophantine type for the linear approximation. The definition
tests the possibility of approximating $0$ by an integer linear combination
of its components.

The notation $\left\vert \left\vert .\right\vert \right\vert $ will indicate
the distance to the nearest integer vector (or number) in $\mathbb{R}$, and $%
|k|=\sup_{0\leq i\leq d}|k_{i}|$ indicates the supremum norm.

\begin{definition}
\label{linapp} The Diophantine type of $\theta =(\theta _{1},...,\theta
_{d}) $ for the linear approximation is 
\begin{equation*}
\gamma _{l}(\theta )=\inf \{\gamma ~,s.t.\,\exists c_{0}>0~s.t.\,\Vert
k\cdot \theta \Vert \geq c_{0}|k|^{-\gamma }~\forall 0\neq k\in \mathbb{Z}%
^{d}\mathbb{\}}.
\end{equation*}
\end{definition}

\subsection{The decay of correlations\label{decorr1}}

In \cite{Ru}, it was observed that piecewise constant toral extensions
cannot have exponential decay of correlations (in \cite{N2} by the way it is
shown that for some piecewise constant $SU_{2}(\mathbb{C})$ extensions there
can be exponential decay of correlations). Quantitative estimates for the
speed of decay of correlations by the arithmetical properties of the angles,
have been given in \cite{GSR}.

In this section we recall those results and see that the systems defined
above have at least polynomial decay of correlations, while for some choice
of the angles the speed of decay is proved to be actually polynomial.

\begin{definition}[Decay of correlations]
\label{def:decorr}Let $\phi ,$ $\psi :X\rightarrow \mathbb{R}$ be
observables on $X$ belonging to the Banach spaces $B,B^{\prime }$, let $\nu $
be an invariant measure for $T$. Let $\Phi :\mathbb{N\rightarrow R}$ such
that $\Phi (n)\underset{n\rightarrow \infty }{\rightarrow }0$.\ A system $%
(X,T,\nu )$ is said to have decay of correlations with speed $\Phi $ with
respect to observables in $B$ and $B^{\prime }$ if 
\begin{equation}
\left\vert \int \phi \circ T^{n}\psi d\nu -\int \phi d\nu \int \psi d\nu
\right\vert \leq \left\vert \left\vert \phi \right\vert \right\vert
_{B}\left\vert \left\vert \psi \right\vert \right\vert _{B^{\prime }}\Phi (n)
\label{decorr}
\end{equation}%
where $||~||_{B}$,$||~||_{B^{\prime }}$ are the norms in $B$ and $B^{\prime
} $.
\end{definition}

The decay of correlations depends on the class of observables considered. On
the skew products satisfying conditions $Te1$ and $Te2$ as above, it is
possible to establish an explicit upper bound for the rate of decay of
correlations which depend on the linear type of the translation angle (see 
\cite{GSR} , Lemma11).

\begin{proposition}
\label{pro:doc} In the piecewise constant toral extensions described above,
for Lipschitz observables the rate of decay is 
\begin{equation*}
\Phi (n)=O(n^{-\frac{1}{2\gamma }})
\end{equation*}%
for any $\gamma >\gamma _{l}(\theta )$.

For $C^{p}$, $C^{q}$ observables, the rate of decay is%
\begin{equation*}
\Phi (n)=O(n^{-\frac{1}{2\gamma }\max (p,q,p+q-d)})
\end{equation*}%
for any $\gamma >\gamma _{l}(\theta )$.
\end{proposition}

\begin{remark}
We remark that the rate is actually polyomial in some cases. In \cite{GSR},
Section 5 (using a result of \cite{GP}) it is proved that if the Diophantine
type is large, then the mixing rate of the systems we consider is actually
slow, with a power law speed which depends on the Diophantine type. In a
system satisfying (\ref{decorr}), let the exponent of power law decay be
defined by 
\begin{equation*}
p=\lim \inf_{n\rightarrow \infty }\frac{-\log \Phi (n)}{\log n}.
\end{equation*}%
Let us consider the skew product of the doubling map and a circle rotation
endowed with the Lebesgue (invariant) measure. For this example the exponent 
$p$ satisfies 
\begin{equation*}
\frac{1}{2\gamma (\theta )}\leq p\leq \frac{6}{\max (2,\gamma (\theta ))-2}.
\end{equation*}
\end{remark}

\subsection{Convergence to equilibrium\label{conv}}

We will use the decay of correlations of the toral extensions to get a
convergence to equilibrium result with respect to the strong and weak norm
of our anisotropic spaces. We have from Proposition \ref{pro:doc} \ that for
Lipschitz observables the rate of decay is $O(n^{-\frac{1}{2\gamma }})$ for
any $\gamma >\gamma _{l}(\alpha )$ and for any Lipschitz observables with $%
\int f=0:$%
\begin{equation*}
|\int g\circ F^{n}~f~dm|\leq C||f||_{lip}||g||_{lip}n^{^{-\frac{1}{2\gamma }%
}}.
\end{equation*}

From this we will prove that 
\begin{equation*}
||L^{n}\mu ||_{"1"}\leq C_{4}n^{^{-\frac{1}{8\gamma }}}||\mu ||_{1-BV}.
\end{equation*}%
For this purpose our strategy is to approximate a $1-BV$ measure $\mu $
which is meant to be iterated with a Lipschitz density and use the decay of
correlations with Lipschitz observables to estimate its convergence to
equilibrium. We remark that a statement of this kind extend greatly the
kinds of measures which are meant to be iterated, as the space of $1-BV$
measures contains measures with singular behavior in the neutral direction.

The first step \ in the strategy is approximating the disintegration of $\mu 
$ with a kind of "piecewise constant one" in next Lemma.

\begin{lemma}
\label{prevv}Let us consider a uniform grid of size $\epsilon =\frac{1}{m}%
,m\in \mathbb{N}$, on the interval $[0,1]$. Given a measure $\mu $ with $%
||\mu ||_{1-BV}<\infty $. There is a measure $\mu _{\epsilon }$ such that $%
\mu _{\epsilon }$ is piecewise constant on the $\epsilon $-grid \NEG{(}$\mu
_{\epsilon }|_{x}$ is constant on each element of the grid as $x$ varies) and%
\begin{equation*}
var_{1}(\mu _{\epsilon })\leq 2var_{1}(\mu ),~||\mu _{\epsilon }||_{1}\leq
||\mu ||_{1}
\end{equation*}

furthermore suppose%
\begin{equation*}
||\mu -\mu _{\epsilon }||_{1}\leq 2\epsilon var_{1}(\mu ).
\end{equation*}
\end{lemma}

\begin{proof}
Let us consider $\mu _{\epsilon }$ defined by averaging in the following
way: let $x\in \lbrack 0,1]$ and $I_{i}$ be the element of the $\epsilon $%
-grid containing $x.$ Then for a measurable set $A\subseteq M$ $\mu |_{x}(A)$
is defined as%
\begin{equation*}
\mu _{\epsilon }|_{x}(A)=\int_{I_{i}}\mu _{\epsilon }|_{\gamma }(A)d\gamma .
\end{equation*}

We remark that $\mu |_{x}-\mu _{\epsilon }|_{x}\leq osc(\epsilon
,x_{i}(x),\mu )$ \ where $x_{i}(x)$ is the grid center closest to $x$, and $%
osc(\epsilon ,x_{i}(x),\mu )\leq osc(2\epsilon ,x,\mu )$ \ then%
\begin{eqnarray*}
\int ||\mu |_{x}-\mu _{\epsilon }|_{x}||_{W} &\leq &2\epsilon \frac{\int
osc(2\epsilon ,x,\mu )}{2\epsilon } \\
&\leq &2\epsilon \sup_{2\epsilon \leq A}(\frac{\int osc(2\epsilon ,x,\mu )}{%
2\epsilon }) \\
&\leq &2\epsilon var_{p}(\mu )
\end{eqnarray*}

The other inequalities are analogous.
\end{proof}

\begin{proposition}
The convergence to equilibrium of a system satisfying Te1, Te2 can be
estimated as%
\begin{equation*}
||L^{n}\nu ||_{1}\leq C_{4}n^{^{-\frac{1}{8\gamma }}}||\nu ||_{1-BV}.
\end{equation*}
\end{proposition}

\begin{proof}
Consider a $1-BV$ measure $\nu $, without loss of generality we can suppose $%
||\nu ||_{1-BV}=1.$ Let us approximate $\nu $ it with a Lipschitz measure.
First let us approximate it with a piecewise constant \ measure $\nu
_{\epsilon }$ as before. We have%
\begin{equation*}
||\nu -\nu _{\epsilon }||_{1}\leq 2\epsilon var_{1}(\nu )\leq 2\epsilon .
\end{equation*}%
Let $\nu _{i}$ be such that $\nu _{i}=\nu _{\epsilon }|_{x_{i}}$ with $x_{i}$
center of the $\epsilon $ grid as before, and $f_{i}$ be the convolution $%
\gamma \ast \nu _{i}$ where $\gamma $ is a $\epsilon _{2}^{-1}$ Lipschitz
mollifier supported in $[-\epsilon _{2},\epsilon _{2}]^{d}$. $f_{i}$ is a $%
\epsilon _{2}^{-1}$ Lipschitz function. Let%
\begin{equation*}
f(x,y)=\left\{ 
\begin{array}{c}
f_{i}(y)~if~|x-x_{i}|\leq (1-\epsilon _{2})\epsilon \\ 
\phi _{i}(x)f_{i}(y)+(1-\phi _{i}(x))f_{i+1}(y)~if~x_{i}+(1-\epsilon
_{2})\epsilon \leq x\leq x_{i+1}-(1-\epsilon _{2})\epsilon%
\end{array}%
\right.
\end{equation*}%
where $\phi _{i}$ is a linear function s.t. $\phi _{i}(x_{i}+(1-\epsilon
_{2})\epsilon )=0$ and $\phi _{i}(x_{i+1}-(1-\epsilon _{2})\epsilon )=1$. We
remark that $f$ is $\sqrt{2}\epsilon _{2}^{-1}\epsilon ^{-1}$ Lipschitz, $%
\int f~dm=0$. and $||\nu _{\epsilon }-fm||_{"1"}\leq 3\epsilon _{2}.$ Hence 
\begin{equation}
||\nu -fm||_{"1"}\leq 2\epsilon var_{1}(\nu )+3\epsilon _{2}.  \label{5665}
\end{equation}

Since the convolution with a Lipschitz kernel is a weak contraction in the
Wasserstein norm, applying Lemma \ref{prevv} we get $var_{1}(fm)\leq
2var_{1}(\nu )$ and $||f||_{"1"}\leq ||\nu ||_{"1"}$. Now we apply
Proposition \ref{pro:doc} in an efficient way. The proposition concerns the
behavior of the correlation of two observables. We will consider $f$ as one
of them, and the other will be constructed in a suitable way to get the
desired statement.

Let $f$ be the Lipschitz density found above. Let $\mu =L^{n}fm$. Let $\mu
_{\epsilon }\ $its piecewise constant approximation defined as in Lemma \ref%
{prevv} and$\ \mu _{i}=\mu _{\epsilon }|_{x_{i}}$. Consider 1-Lipschitz
functions $l_{i}:\mathcal{T}^{d}\rightarrow \mathbb{R}$ such that $|~|\int
l_{i}\mu _{i}|-||\mu _{i}||_{W}~|\leq \xi $, consider functions $%
h_{i}:[0,1]\rightarrow \mathbb{R}$ such that $h_{i}=1$ on the central third
of the $i$ interval of the $\epsilon $-grid and zero elsewhere, and $%
lip(h_{i})=3\epsilon ^{-1}$. Consider $g_{i}:X\rightarrow \mathbb{R}$\
defined by $g_{i}(x,y)=$ $l_{i}(y)h_{i}(x)$ and $g=\sum_{i}g_{i}$. By what
is said above%
\begin{equation*}
||\mu _{\epsilon }||_{1}\leq 3(\xi +\int g\mu _{\epsilon })
\end{equation*}%
and by Lemma \ref{prevv}, $||L^{n}fm-\mu _{\epsilon }||_{1}\leq 2\epsilon
var_{1}(\mu )$. Then%
\begin{eqnarray*}
||L^{n}(fm)||_{1} &\leq &||\mu _{\epsilon }||_{1}+2\epsilon var_{1}(\mu ) \\
&\leq &3(\xi +\int g\mu _{\epsilon })+2\epsilon var_{1}(\mu ).
\end{eqnarray*}

Now consider $\int g~L^{n}fm.$ Since $g$ is 1-Lipschitz in the $y$
direction, we have that 
\begin{equation*}
|\int g~L^{n}f-\int g~\mu _{\epsilon }|\leq ||L^{n}fm-\mu _{\epsilon
}||_{1}\leq 2\epsilon var_{1}(\mu )
\end{equation*}%
and \ 
\begin{equation*}
||L^{n}f||_{1}\leq 3(\xi +\int g~dL^{n}f\mu _{0}+2\epsilon var_{1}(\mu
))+2\epsilon var_{1}(\mu ).
\end{equation*}

Now, since $\int fdm=0$, by \ Proposition \ref{pro:doc}%
\begin{equation*}
|\int g~dL^{n}(fm)|\leq C||f||_{lip}||g||_{lip}n^{^{-\frac{1}{2\gamma }}}
\end{equation*}%
then%
\begin{equation*}
||L^{n}f||_{1}\leq 3(\xi +C||f||_{lip}3\epsilon ^{-1}n^{^{-\frac{1}{2\gamma }%
}}+2\epsilon var_{1}(L^{n}(fm)))+2\epsilon var_{1}(L^{n}(fm)).
\end{equation*}

We recall that the Lebesgue measure is invariant for the system. Then if $K$
is a constant density such that $f+K\geq 0$ it holds $%
L^{n}(fm+Km)=L^{n}(fm)+Km.$ It holds $var_{1}(L^{n}f)=var_{1}(L^{n}(f+K))$,
since it is a positive measure, to $(f+K)m$ we can apply the regularization
inequality. Setting $B=\frac{B_{T}({\lambda }^{p}+{\hat{H}}+3qC_{h})}{1-{%
\lambda }^{p}{\alpha }}$ we get%
\begin{equation}
var_{1}(L^{n}(fm)){\leq \lambda }^{n}{\alpha }^{n}{~var}_{1}{(f)+}%
B(||f||_{"1"}+K)
\end{equation}%
since $f$ is $\sqrt{2}\epsilon _{2}^{-1}\epsilon ^{-1}$-Lipschitz and $%
||f||_{1-BV}\leq 1$, $||f||_{\infty }\leq \sqrt{2}\epsilon _{2}^{-1}\epsilon
^{-1}+1$, then $var(L^{n}(fm)){\leq \lambda }^{n}{\alpha }^{n}{~var(f)+}B(1+%
\sqrt{2}\epsilon _{2}^{-1}\epsilon ^{-1}+1)$ and 
\begin{eqnarray*}
||L^{n}(fm)||_{1} &\leq &3\xi +3C||f||_{lip}3\epsilon ^{-1}n^{^{-\frac{1}{%
2\gamma }}}+8\epsilon var_{1}(L^{n}f) \\
&\leq &3\xi +3C||f||_{lip}3\epsilon ^{-1}n^{^{-\frac{1}{2\gamma }%
}}+8\epsilon \lbrack {\lambda }^{n}{\alpha }^{n}{~var}_{1}{(f)+}B(1+\sqrt{2}%
\epsilon _{2}^{-1}\epsilon ^{-1}+1)] \\
&\leq &3\xi +3C||f||_{lip}3\epsilon ^{-1}n^{^{-\frac{1}{2\gamma }%
}}+16\epsilon \lbrack {\lambda }^{n}{\alpha }^{n}{~var}_{1}{(f)+}B(1+\sqrt{2}%
\epsilon _{2}^{-1}\epsilon ^{-1}+1)] \\
&\leq &3\xi +3C\sqrt{2}\epsilon _{2}^{-1}\epsilon ^{-1}3\epsilon ^{-1}n^{^{-%
\frac{1}{2\gamma }}}+16\epsilon \lbrack {\lambda }^{n}{\alpha }^{n}{~var}_{1}%
{(f)+}B(1+\sqrt{2}\epsilon _{2}^{-1}\epsilon ^{-1}+1)]
\end{eqnarray*}

Taking $\xi =n^{^{-\frac{1}{2\gamma }}},$ $\epsilon _{2}=n^{^{-\frac{1}{%
8\gamma }}},~\epsilon =n^{^{-\frac{1}{8\gamma }}}$ recalling that ${var}_{1}{%
(f)\leq var}_{1}(\nu )\leq 1$%
\begin{eqnarray*}
||L^{n}(fm)||_{1} &\leq &3n^{^{-\frac{1}{2\gamma }}}+9C\sqrt{2}n^{^{\frac{3}{%
8\gamma }}}n^{^{-\frac{1}{2\gamma }}}+16n^{^{-\frac{1}{8\gamma }}}(\alpha
\lambda ^{n}+2B)+\sqrt{2}Bn^{^{-\frac{1}{8\gamma }}} \\
&\leq &C_{3}n^{^{-\frac{1}{8\gamma }}}.
\end{eqnarray*}

Finally, by Equation \ref{5665} 
\begin{eqnarray*}
||L^{n}\nu ||_{1} &\leq &||L^{n}(\nu -fm)||_{"1"}+||L^{n}(fm)||_{"1"} \\
&\leq &2\epsilon ||\nu ||_{1-BV}+||L^{n}(fm)||_{"1"}+3\epsilon _{2} \\
&\leq &C_{4}n^{^{-\frac{1}{8\gamma }}}.
\end{eqnarray*}
\end{proof}

Once we have an estimate for the speed of convergence to equilibrium, by
Proposition \ref{thm}, and Remark \ref{k1}, the following holds directly:

\begin{proposition}
\label{28}Consider a family of skew product maps $F_{\delta }=(T_{\delta
},G_{\delta }),$ $0\leq \delta \leq D$ satisfying uniformly $Sk1,...,Sk3$
and let $f_{\delta }\in \mathcal{L}^{1}$ its invariant probability measures,
suppose

\begin{enumerate}
\item $F_{0}$ is a piecewise constant toral extension as defined in Section %
\ref{Sec2}, with linear Diophantine type $\gamma _{l}(\theta );$

\item There is $n\in \mathbb{N}$ such that for each $\delta \leq
D~,~d_{n,S}(T_{0},T_{\delta })\leq \delta ;$

\item for each $\delta \leq D$ there is a set $A_{2}\in Int_{n}$ such that $%
m(A_{2})\geq 1-\delta $ and for all $x\in A_{2},y\in \mathcal{T}^{d}:$ $%
|G_{0}(x,y)-G_{\delta }(x,y)|\leq \delta .$
\end{enumerate}

Then for each $\gamma >\gamma _{l}(\theta )$ there is $K_{1}$ such that for $%
\delta $ small enough%
\begin{equation*}
||f_{\delta }-f_{0}||_{"1"}\leq K_{1}\delta ^{\frac{1}{8\gamma +1}}.
\end{equation*}
\end{proposition}

\subsection{An example having H\"{o}lder behavior\label{slat}}

In this section we show a simple example of perturbation of toral extensions
satisfying assumptions $Te1$ and $Te2$ for which the statistical behavior is
actually, only H\"{o}lder stable. This shows how that Propositions \ref{thm}
and \ref{28} give a general estimate, which is quite sharph in the case of
piecewise constant toral extensions.

\begin{proposition}
\label{bahh}Consider a well approximable Diophantine irrational $\theta $
with $\gamma _{l}(\theta )>2$. Let us consider the map $F_{0}:[0,1]\times 
\mathcal{T}^{1}$ defined as a skew product $F_{0}(T_{0}(x),G_{0}(x,y))$ where%
\begin{equation*}
T_{0}(x)=2x~\func{mod}(1)
\end{equation*}%
and%
\begin{equation*}
G_{0}(x,y)=~y+\theta \varphi (x)
\end{equation*}%
where $\varphi =\chi _{\lbrack \frac{1}{2},1]}$. Consider $\gamma ^{\prime
}<\gamma _{l}(\theta )$; there is there is a sequence of reals $\delta
_{j}\geq 0$, $\delta _{j}\rightarrow 0$ and a sequence of perturbed of maps $%
\hat{F}_{\delta _{j}}(x,y)=(\hat{T}_{\delta _{j}}(x),\hat{G}_{\delta
_{j}}(x,y))$ satisfying $Sk1,..,Sk3,$ with $\hat{T}_{\delta
_{j}}(x)=T_{0}(x) $ and $||\hat{G}_{\delta _{j}}(x,y)-G_{0}(x,y)||_{\infty
}\leq 2\delta _{j}$ such that%
\begin{equation*}
||\mu _{0}-\mu _{j}||_{"1"}\geq \frac{1}{9}\delta _{j}{}^{\frac{1}{\gamma
^{\prime }-1}}
\end{equation*}%
holds for every $j$ and every $\mu _{j}$, invariant measure of $\hat{F}%
_{\delta _{j}}(x,y)$ in $\mathcal{L}^{1}$.
\end{proposition}

\begin{proof}
We remark that since there is convergence to equilibrim for $F_{0},$ the
Lebesgue measure $\mu _{0}$ on $[0,1]\times \mathcal{T}^{1}$ is the unique
invariant measure in $\mathcal{L}^{1}$ for $F_{0}.$ Consider $F_{\delta
}=(T_{0}(x),~y+(\delta +\theta )\varphi (x)))$. For a sequence of values of  
$\delta $ converging to $0$ it holds that $(\delta +\theta )$ is rational.
For this sequence the map $y\rightarrow y+(\delta +\theta )$ ($:\mathcal{T}%
^{1}\rightarrow \mathcal{T}^{1}$ ) is such that, $0$\ has a periodic orbit.
Let $y_{1}=0,...,y_{k}$ be this orbit. For these parameters, consider the
product measure $\mu _{n}=\frac{1}{k}\sum_{i\leq k}m\otimes \delta _{y_{i}}$%
, where $m$ is the Lebesgue measure on $[0,1]$ and $\delta _{y_{i}}$ is the
delta measure placed on $y_{i}$. \ The measure $\mu _{n}$ is invariant for $%
F_{\delta }(x,y).$ and is in $\mathcal{L}^{1}$. It is easy to see that $%
||\mu _{0}-\mu _{n}||_{"1"}\geq \frac{1}{9}\frac{1}{k}$. Now the Diophantine
type of $\theta $ will give an estimate for the relation between $\delta $
and $k$. Indeed let $\gamma ^{\prime }<\gamma (\theta )$, by the Diophantine
type of $\theta $ we know that there are infinitely many $k_{j}$ and
integers $p_{j}$ such that $|k_{j}\theta -p_{j}|\leq |\frac{1}{k_{j}}%
|^{\gamma ^{\prime }}$ then $|\theta -\frac{p_{j}}{k_{j}}|\leq |\frac{1}{%
k_{j}}|^{\gamma ^{\prime }-1}$. Let us now consider $\delta _{j}=-\theta +%
\frac{p_{j}}{k_{j}}$, it holds $|\delta _{j}|\leq |\frac{1}{k_{j}}|^{\gamma
^{\prime }-1}$ and the angle $(\delta _{j}+\theta )$ generates a periodic
orbit of period $k_{j}$. This happens by perturbing the second coordinate of
the map by a quantity which is less than $|\frac{1}{k_{j}}|^{\gamma ^{\prime
}-1}$.\ Summarizing, for the map $F_{\delta _{j}}$ we have that there is no
perturbation on the first coordinate, for the second coordinate, $%
||G_{0}-G_{\delta _{j}}||_{\infty }\leq \delta _{j}$ and denoting as $\mu
_{j}$ the invariant measure on the periodic orbit defined before it holds%
\begin{equation*}
||\mu _{0}-\mu _{j}||_{"1"}\geq \frac{1}{9}\delta _{j}{}^{\frac{1}{\gamma
^{\prime }-1}}.
\end{equation*}

This example can further be improved by perturbing the map $F_{\delta _{j}}$
to a new map $\hat{F}_{\delta _{j}}$ in a way that $\mu _{j}$ (a measure
supported on the attractor of $\hat{F}_{\delta _{j}}$) and $\mu _{j}+\frac{%
k_{j}}{2}\footnote{%
Defined as $[\mu _{j}+\frac{1}{2k_{j}}](A)=\mu _{j}(A-\frac{1}{2k_{j}})$ \
for each measurable set $A$ in $\mathcal{T}^{1}$. Where $A-\frac{1}{2k_{j}}$
is the translation of the set $A$ by $-\frac{1}{2k_{j}}$ .}$ (supported on
the repeller of $\hat{F}_{\delta _{j}}$) are the only invariant measures in $%
\mathcal{L}^{1}$ for $\hat{F}_{\delta _{j}}$ and $\mu _{j}$ is the unique
physical measure for the system. This can be done by making a small further $%
C^{\infty }$ perturbation on $G$. Let us denote again by $%
y_{1},...,y_{k_{j}} $\ the periodic orbit of $0$\ as before. Let us consider
a $C^{\infty }$ function $g:[0,1]\rightarrow \lbrack 0,1]$ such that:

\begin{itemize}
\item $g$ is negative on the each interval $[y_{i},y_{i}+\frac{1}{2k_{j}}]$
and positive on each interval $[y_{i}+\frac{1}{2k_{j}},y_{i+1}]$ (so that $%
g(y_{i}+\frac{1}{2k_{j}})=0$ );

\item $g^{\prime }$ is positive in each interval $[y_{i}+\frac{1}{3k_{j}}%
,y_{i+1}-\frac{1}{3k_{j}}]$ and negative in $[y_{i},y_{i+1}]-[y_{i}+\frac{1}{%
3k_{j}},y_{i+1}-\frac{1}{3k_{j}}]$.
\end{itemize}

Considering $D_{\delta }:\mathcal{T}^{1}\rightarrow \mathcal{T}^{1}$ defined
by $D_{\delta }(x)=x+\delta g(x)$~$\func{mod}(1)$ \ it holds that the
iteration of this map send all the space but the set $\{y_{i}+\frac{1}{2k_{j}%
}~s.t.~i\leq k_{j}\}$ (which is a repeller) to the set $\{y_{i}~s.t.~i\leq
k_{j}\}$ (the attractor). Then define $\hat{F}_{\delta _{j}}$ as:%
\begin{equation*}
\hat{F}_{\delta _{j}}(x,y)=(T_{\delta _{j}}(x),~D_{\delta _{j}}(y+(\delta
_{j}+\theta )\varphi (x))).
\end{equation*}

The claim directly follows from the remark that for the map $(\hat{F}%
_{\delta _{j}})^{k_{j}}$ the sets \ $\Gamma _{1}:=[0,1]\times
\{y_{i}~s.t.~i\leq k_{j}\}$ and $\Gamma _{2}:=[0,1]\times \{y_{i}+\frac{1}{%
2k_{j}}~s.t.~i\leq k_{j}\}$ are invariant and the set $\Gamma _{1}$ attracts
the whole $[0,1]\times \mathcal{T}^{1}-\Gamma _{2}$.
\end{proof}

The construction done in the previous proof can be extended to show H\"{o}%
lder behavior for the average of a given regular observable. We show an
explicit example of such an observable for a skew product with a particular
angle $\theta $.

\begin{proposition}
\label{30}Consider a map $F_{0}$ as above with the rotation angle $\theta
=\sum_{1}^{\infty }2^{-2^{2i}}$. With%
\begin{equation*}
T_{0}(x)=2x~\func{mod}(1)
\end{equation*}%
and%
\begin{equation*}
G_{0}(x,y)=~y+\theta \varphi (x)
\end{equation*}%
as in Proposition \ref{bahh}. Let $\hat{F}_{\delta _{j}}$ be its
perturbations as described in the proof of the proposition and $\mu _{j}$
their invariant measures in $\mathcal{L}^{1}.$ There is an observable $\psi
:[0,1]\times \mathcal{T}^{1}\rightarrow \mathbb{R}$ \ with derivative in $%
L^{2}$ and $C\geq 0$ such that 
\begin{equation*}
|\int \psi d\mu _{0}-\int \psi d\mu _{j}|\geq C\sqrt{\delta _{j}}.
\end{equation*}
\end{proposition}

\begin{proof}
We recall that that $\sum_{n+1}^{\infty }2^{-2^{2i}}\leq 2^{-2^{2(n+1)}+1}$.
By this $||2^{2^{2n}}\theta ||\leq 2^{-2^{2(n+1)}+1}$ and the Diophantine
type of $\theta $ is greater than $4$. Following the construction above, we
have that with a perturbation of size less than $2^{-2^{2(n+1)}+1}$ the
angles $\theta _{j}=\sum_{1}^{j}2^{-2^{2i}}$ generate on the second
coordinate of the skew product orbits of period $2^{2^{2j}}$. Now let us
construct a suitable observable which can "see" the change of the invariant
measure under this perturbation. Let us consider 
\begin{equation}
\psi (x,y)=\sum_{1}^{\infty }\frac{1}{(2^{2^{2i}})^{2}}\cos (2^{2^{2i}}2\pi
y)  \label{obss}
\end{equation}
and $\psi _{k}(x,y)=\sum_{1}^{k}\frac{1}{(2^{2^{2i}})^{2}}\cos
(2^{2^{2i}}2\pi y)$ Since for the observable $\psi $, the $i$-th Fourier
coefficient decreases like $i^{-2}$, then $\psi $ has a derivative in $L^{2}$%
. Let $x_{1}=0,...,x_{2^{2^{2j}}}$ be the periodic orbit of $0$ for $%
y\rightarrow y+\theta _{j}$, and $\mu _{j}=\frac{1}{2^{2^{2i}}}\sum \delta
_{x_{i}}$ the physical measure supported on it. Since $2^{2^{2j}}$\ divides $%
2^{2^{2(j+1)}}$ then\ $\sum_{i=1}^{2^{2^{2j}}}\psi _{k}(x_{i})=0$ for every $%
k<j$, thus $\int \psi _{j-1}~d\mu _{j}=0.$ Then 
\begin{eqnarray*}
v_{j} &:&=\int \psi ~d\mu _{j}\geq \frac{1}{(2^{2^{2j}})^{2}}%
-\sum_{j+1}^{\infty }\frac{1}{(2^{2^{2i}})^{2}} \\
&\geq &2^{-2^{2j+1}}-2^{-2^{2(j+1)}+1}.
\end{eqnarray*}

And for $j$ big enough%
\begin{equation*}
2^{-2^{2j+1}}-2^{-2^{2(j+1)}+1}\geq \frac{1}{2}(2^{-2^{2j}})^{2}.
\end{equation*}%
Summarizing, with a perturbation of size $\delta _{j}=\sum_{j+1}^{\infty
}2^{-2^{2i}}\leq 2\ast 2^{-2^{2(j+1)}}=2^{-2^{2(j+1)}}=2(2^{-2^{2j}})^{4}$
we get a change of average for the observable $\psi $ from $\int \psi dm=0$
to $v_{n}\geq \frac{1}{2}(2^{-2^{2j}})^{2}.$ Hence we have that there is a $%
C\geq 0$ such that with a perturbation of size $\delta _{j}$ we get a change
of average for the observable $\psi $ of size bigger than $C\sqrt{\delta _{j}%
}.$
\end{proof}

\begin{remark}
Using $\frac{1}{(2^{2^{2i}})^{\alpha }}$ instead of $\frac{1}{%
(2^{2^{2i}})^{2}}$ in (\ref{obss} ) we can obtain a smoother observable.
Using rotation angles with bigger and bigger Diophantine type it is possible
to obtain a dependence of the physical measure to perturbations with worse
and worse H\"{o}lder exponent. Using angles with infinite Diophantine type
it is possible to have a behavior whose modulus of continuity is worse than
the H\"{o}lder one.
\end{remark}

\end{document}